\documentclass[a4paper,11pt]{article}
\usepackage[margin=1in]{geometry}
\usepackage{xfrac}
\usepackage{amsthm}
\usepackage{graphicx}
\usepackage{latex-macros}
\usepackage{amssymb}
\usepackage{appendix}
\usepackage{cancel}
\usepackage{enumerate}
\usepackage{booktabs}
\usepackage{multirow}
\usepackage{algorithm}
\usepackage[noend]{algpseudocode}
\usepackage[dvipsnames]{xcolor}
\usepackage{hyperref}
\usepackage{subcaption}
\usepackage[labelfont=bf]{caption}
\usepackage{aligned-overset}
\hypersetup{
	colorlinks=true,
  citecolor={Violet},
  linkcolor={red!50!black},
	urlcolor=Magenta}

\usepackage[round,authoryear]{natbib}

\usepackage[noabbrev,capitalize]{cleveref}

\definecolor{lblue}{HTML}{908cc0}
\definecolor{mblue}{HTML}{519cc8}
\definecolor{hblue}{HTML}{1d5996}
\definecolor{lred}{HTML}{cb5501}
\definecolor{mred}{HTML}{f1885b}
\definecolor{hred}{HTML}{b3001e}
\definecolor{ttred}{HTML}{ca3542}

\newcommand{\deuc}{d_{\mathsf{euc}}}
\newcommand{\diam}{\mathbf{diam}}
\newcommand{\indep}{\mathrel{\perp\!\!\!\!\perp}}

\theoremstyle{definition}
\newtheorem{assumption}{Assumption}
\crefname{assumption}{Assumption}{Assumptions}

\theoremstyle{proposition}
\newtheorem*{theorem*}{Theorem}

\crefname{appsec}{Appendix}{Appendices}
\crefname{claim}{Claim}{Claims}
\crefname{lemma}{Lemma}{Lemmas}
\Crefname{lemma}{Lemma}{Lemmas}

\newcommand{\sol}{x_{\star}}
\newcommand{\lip}{\mathsf{L}}
\renewcommand{\cite}{\citep}

\newcommand{\strictind}{\mathcal{I}}
\newcommand{\slow}{\mathsf{slow}}
\newcommand{\rhoslow}{\overline{\rho}}
\newcommand{\rhofast}{\rho}
\newcommand{\fast}{\mathsf{fast}}
\newcommand{\condnum}{\kappa}
\newcommand{\condest}{\widehat{\kappa}}

\usepackage{tikz,pgfplots}
\pgfplotsset{compat=1.14}
\usetikzlibrary{automata,decorations.markings,positioning}

\usepackage{pifont}

\allowdisplaybreaks

\usepackage{authblk}

\title{Nonlinear tomographic reconstruction via nonsmooth optimization}
\author[1]{Vasileios Charisopoulos\thanks{\email{vchariso@uchicago.edu}}}
\author[1,2,3]{Rebecca Willett\thanks{\email{willett@uchicago.edu}}}
\affil[1]{Data Science Institute, University of Chicago}
\affil[2]{Department of Computer Science, University of Chicago}
\affil[3]{Computational and Applied Mathematics, Department of Statistics, University of Chicago}

\begin{document}

\maketitle

\begin{abstract}
	We study iterative signal reconstruction in computed tomography (CT), wherein
	measurements are produced by a linear transformation of the unknown signal
	followed by an exponential nonlinear map. Approaches based on pre-processing the
	data with a log transform and then solving the resulting linear inverse problem are
	tempting since they are amenable to convex optimization methods; however, such
	methods perform poorly when the underlying image has high dynamic range, as in X-ray
	imaging of tissue with embedded metal. We show that a suitably initialized
	subgradient method applied to a natural nonsmooth, nonconvex loss function
	produces iterates that converge to the unknown signal of interest at a
	geometric rate under the statistical model proposed by~\citet{FVW+23}.
	Our recovery program enjoys improved conditioning compared to the formulation
	proposed by the latter work, enabling faster iterative reconstruction from substantially fewer samples.
\end{abstract}

\section{Introduction}

A computed tomography (CT) scan is an imaging procedure in which
a motorized X-ray source rotates rapidly around an object of interest, emitting
narrow beams of X-rays at the object at varying angles. Detectors placed on
the opposite side of the X-ray source measure the transmitted radiation, and
the measurements are then combined to synthesize cross-sectional images of the
object. In the context of medical imaging, for instance, the resulting images
can be used for diagnostic purposes by clinicians; beyond the medical setting,
CT imaging is also widely used in security settings~\cite{TSA},
non-destructive material evaluation~\cite{NDE},
archaeology~\cite{CB08}, and more.
We refer the reader to~\citet{Buzug11} for an overview of modern CT technologies.

Since the final images in CT scans are constructed from X-ray
measurements at different angles, the number of collected measurements can
heavily influence the quality of the final reconstruction. Indeed, reconstruction
artifacts in medical CT scans can obstruct anatomical features that are important for therapy
planning. However, large doses of radiation can have adverse effects on living
tissue, highlighting the need for sample-efficient imaging techniques that use
few measurements. At the same time, rapid reconstruction methods can increase
the efficiency of radiology labs and reduce time-to-diagnosis.

Several commercial CT scanners rely on the following simple physical model~\citep{Natterer2001}: let $u \in \Rbb^2$ denote a spatial index of the tissue being imaged,
and suppose that the X-ray attenuation coefficient is given by a (unknown) function
$\mathscr{I}: \Rbb^2 \to \Rbb_+$. If the initial beam intensity is $I_{0}$,
the measured beam intensity across line $\ell$, $y_{\ell}$, satisfies
\begin{equation}
	y_{\ell} = I_{0} \expfun{-\int_{\ell} \mathscr{I}(u) \dd{u}}.
	\label{eq:physical-model}
\end{equation}
The goal of computed tomography is to recover the attenuation coefficients from
measurements of the form~\eqref{eq:physical-model}. In practice, we have access
to a limited number of measurements corresponding to a finite set of lines, indexed
by a set of known angles $\set{\theta_1, \dots, \theta_m}$. In this case, we write
\begin{equation}
	y_{i} = I_0 \expfun{-\ip{a_i, \sol}}, \quad \text{for $i = 1, \dots, m$,}
	\label{eq:physical-model-discretized}
\end{equation}
where each $a_i \in \Rbb^{d}$, $i = 1, \dots, m$ is a set of known coefficients derived from the Radon transform~\cite{Radon1917}
at angle $\theta_i$
and $\sol \in \Rbb^d$ is a vector representation of the unknown image\footnote{While the unknown signal
	is typically 2D or 3D, we work with the vectorized representation for simplicity.};
i.e., the $j^{\text{th}}$ element of $\sol$ is $\mathscr{I}(u_{j})$, where $u_{j}$ is a discrete pixel location.
To recover $\sol$, commercial CT software applies a logarithmic preprocessing step to the $y_i$~\cite{Fu2017},
resulting in a set of linear equations:
\begin{equation}
	\ip{a_i, \sol} = -\hat{y}_i, \;\; i = 1, \dots, m, \quad \text{where} \quad
	\hat{y}_i = -\log\left(\frac{y_i}{I_0}\right).
	\label{eq:physical-model-linearized}
\end{equation}
The linear inverse problem in~\eqref{eq:physical-model-linearized} is then solved
using traditional methods, most commonly the so-called \emph{filtered back-projection} (FBP)
method. However, the logarithmic transform is numerically unstable as $y_i \to 0$,
which is often the case for X-rays passing through high-density materials, and is
known to lead to reconstruction artifacts~\cite{GMJ+16}. To understand why this would impede recovery,
consider the behavior of \eqref{eq:physical-model-linearized}
as $\norm{\sol}$ grows; the measurements $y_i \to 0$ and the logarithmically transformed $\hat{y}_i$ in~\eqref{eq:physical-model-linearized} approaches $-\infty$.
Therefore, the right-hand side of the linear system in~\eqref{eq:physical-model-linearized}
is sensitive to small perturbations in $y_i$ (e.g., induced by numerical round-off errors). Since the Radon transform yields an
ill-conditioned linear system, this means that we need many more samples to recover
$\sol$ in a stable manner when the norm $\norm{\sol}$ is larger. This motivates researchers and practitioners
to consider iterative reconstruction methods that operate directly on
\eqref{eq:physical-model-discretized}.

Recently,~\citet{FVW+23} studied iterative reconstruction from CT measurements
under a Gaussian measurement model, intended to capture randomness in ray
directions:
\begin{assumption}[Measurement model]
	\label{assm:measurement-model}
	The measurements satisfy\footnote{
		Since $\ip{a_i, \sol}$ represents
		the line integral in the original model~\eqref{eq:physical-model-discretized} and is
		nonnegative by construction, the proposed model in~\eqref{eq:physical-model-fvw}
		enforces nonnegativity by taking the positive part of the inner product.
	}
	\begin{equation}
		y_i = 1 - \expfun{-\ip{a_i, \sol}_+}, \;\; i = 1, \dots, m, \quad
		\text{for} \quad
		[t]_+ := \max(t, 0),
		\label{eq:physical-model-fvw}
	\end{equation}
	where the sensing vectors are i.i.d.\ Gaussian: $a_{i} \iid \cN(0, I_{d})$, for $i = 1, \dots, m$.
\end{assumption}
Given~\cref{assm:measurement-model}, they propose to recover $\sol$ by optimizing the following nonlinear least-squares objective
using gradient descent:
\begin{align}
	\widehat{x} & = \argmin_{x \in \Rbb^d} \frac{1}{2m}
	\sum_{i = 1}^m \left(y_i - (1 - \expfun{-\ip{a_i, x}_+})\right)^2,
	\label{eq:nonlinear-least-squares}
\end{align}
While working with the measurements directly avoids numerical instability, the
optimization problem in~\eqref{eq:nonlinear-least-squares} is nonconvex;
consequently, iterative methods like gradient descent are not guaranteed to converge
to the solution $\sol$. Moreover, it is possible that the
conditioning of the problem in~\eqref{eq:nonlinear-least-squares} is such that
iterative methods are rendered impractical due to slow convergence.
Nevertheless, \citet{FVW+23} argue that \eqref{eq:nonlinear-least-squares}
can be solved to global optimality and provide a rigorous proof for
their claim under the Gaussian measurement model. Their main findings are
summarized below:
\begin{theorem}[{Informal; adapted from~\citep[Theorem 1]{FVW+23}}]
	\label{theorem:gd-informal}
	Let~\cref{assm:measurement-model} hold and suppose that the number of measurements $m$ satisfies
	$m \gtrsim \frac{e^{c \norm{\sol}}}{\norm{\sol}^2} \cdot d$. Then solving
	\eqref{eq:nonlinear-least-squares} using
	gradient descent with stepsize $\eta$ and careful initialization
	produces a sequence of iterates
	$\set{x_k}_{k \geq 1}$ satisfying
	\begin{equation}
		\norm{x_k - \sol}^2 \leq \left(1 - \eta e^{-5 \norm{\sol}}\right)^k \norm{\sol}^2,
	\end{equation}
	as long as $\eta \lesssim e^{-5 \norm{\sol}}$, with high probability over
	the choice of design vectors.
\end{theorem}
The result of~\cref{theorem:gd-informal} is positive and surprising: it
suggests that, despite nonconvexity, the optimization problem in~\eqref{eq:nonlinear-least-squares}
can be solved reliably and efficiently with a simple first-order method initialized
in a rather simple way: namely, by taking a ``long'' gradient descent step at $x_{0} = \bm{0}$
(see Theorem 1 in~\citet{FVW+23} for a precise statement). However, some important limitations
can be gleaned from~\cref{theorem:gd-informal}:
\begin{description}
	\item[(Sample complexity)] The number of samples, $m$, scales exponentially with the
	      norm of the unknown signal $\sol$, and this norm can be large in settings of
	      interest. This requirement is at odds with our
	      desire to minimize exposure to radiation in practical settings.
	\item[(Conditioning)] The ``condition number'' $\exp(5 \norm{\sol})$ of the optimization problem --- which
	      is a natural measure of difficulty governing the speed of
	      convergence of iterative methods --- also scales exponentially with the
	      norm of $\sol$. As a result, convergence is not guaranteed unless the
	      algorithm uses vanishingly small stepsizes. Consequently, the number
	      of iterations (and computation time) required for reconstruction can be
	      rather large.
\end{description}
A natural question is whether these limitations are essential or can be
sidestepped by choosing a different objective function or reconstruction method. In this
paper, we show that using a nonsmooth loss function --- namely, the least absolute deviation
($\ell_1$) penalty --- yields \textbf{exponential improvements in both sample complexity
	and problem conditioning}. Our main finding is the following:

\vspace{2pt}

\begin{center}
	\itshape
	\begin{minipage}{0.75\textwidth}
		Under the Gaussian measurement model, $\sol$ can be recovered from a \textbf{polynomial} number
		of measurements by solving a nonsmooth ($\ell_1$) recovery program, whose
		conditioning is also \textbf{polynomial} in $\norm{\sol}$.
	\end{minipage}
\end{center}

\vspace{2pt}

Solving the aforementioned recovery program can be accomplished with standard
first-order methods; the main method analyzed in our paper is the subgradient
method using a scaled Polyak step-size~\cite{Polyak69}.\
A quantitative comparison between the guarantees of~\citet{FVW+23} and ours can be
found in~\cref{table:comparison}; there,
``iteration complexity'' indicates the number of iterations needed for the estimation
error to fall below a target level $\varepsilon$ and ``sample complexity'' indicates
the minimum number of samples $m$ needed for the iteration complexity bounds
to be valid.

\begin{table}[h]
	\renewcommand{\arraystretch}{1.25}
	\centering
	\begin{tabular}{l l l} \toprule
		\textbf{Method}                                                                                 & \textbf{Iteration complexity} & \textbf{Sample complexity} \\ \midrule
		\texttt{PolyakSGM} (Alg.~\ref{alg:polyaksgm})                                                   &
		\multirow{2}{*}{$O\left(\norm{\sol}^6 \log\left(\frac{\norm{\sol}}{\varepsilon}\right)\right)$} &
		\multirow{2}{*}{$O(\norm{\sol}^4 \cdot d)$}                                                                                                                  \\
		\texttt{AdPolyakSGM} (Alg.~\ref{alg:adpolyak})                                                  &
		                                                                                                &
		\\ \midrule
		\texttt{GD}~\citep{FVW+23}                                                                      &
		$O\left(\exp(c_1 \norm{\sol}) \log\left(\frac{\norm{\sol}}{\varepsilon}\right)\right)$          &
		$O\left(\frac{\exp(c_2 \norm{\sol})}{\norm{\sol}^2} \cdot d\right)$                                                                                          \\ \bottomrule \\
	\end{tabular}
	\caption{Iteration and sample complexity of iterative reconstruction methods.
		The proposed methods achieve exponential gains in both metrics.
		For a formal statement of our results, see~\cref{theorem:main} (\cref{alg:polyaksgm}) and
		\cref{prop:adpolyak-termination} (\cref{alg:adpolyak}).}
	\label{table:comparison}
\end{table}

In the next section, we formalize the measurement model and loss function used
and provide a high-level overview of the proposed method and the key points of its
convergence analysis.

\subsection{Method overview}
\label{sec:method-overview}
In this section, we provide a high-level overview of our method. Working under~\cref{assm:measurement-model},
we propose optimizing the following nonsmooth ($\ell_1$) penalty:
\begin{equation}
	f(x) := \frac{1}{m} \sum_{i = 1}^m \abs{y_i - h_{i}(x)}, \quad
	h_{i}(x) := 1 - \expfun{-\ip{a_i, x}_+}.
	\label{eq:nonsmooth-loss}
\end{equation}
To optimize $f(x)$ in~\eqref{eq:nonsmooth-loss}, we will use the subgradient
method with Polyak stepsize, starting at $x_{0} = \bm{0}$ (see~\cref{alg:polyaksgm}
for a formal description of the method). Note that the optimal value $f_{\star} = f(\sol) = 0$ when the measurements
are not corrupted by noise.

\begin{algorithm}[h]
	\caption{\texttt{PolyakSGM}: Subgradient method with scaled Polyak stepsize}
	\begin{algorithmic}[1]
		\State \textbf{Input}: $x_0 \in \Rbb^{d}$, scaling $\eta \in (0, 1]$, budget $K \in \Nbb$, optimal value $f_{\star}$.
		\For{$k = 1, 2, \dots, K$}
		\State Set $x_{k+1} := x_k - \eta \cdot \frac{f(x_k) - f_{\star}}{\norm{v_k}^2} v_k$, where $v_k \in \partial f(x_k)$.
		\label{op:polyak-step}
		\EndFor
		\State \Return $x_{K}$.
	\end{algorithmic}
	\label{alg:polyaksgm}
\end{algorithm}
In~\cref{alg:polyaksgm}, $\partial f(x)$ denotes
the so-called \emph{Clarke} subdifferential~\cite{Clarke1975}, an appropriate generalization of the convex
subdifferential for arbitrary Lipschitz functions. Like its convex counterpart, the
Clarke subdifferential is set-valued mapping in general. The Clarke subdifferential
of the loss function in~\eqref{eq:nonsmooth-loss} can be calculated
with the help of the chain rule:
\begin{subequations}
	\begin{align}
		\partial f(x)               & =
		\frac{1}{m} \sum_{i = 1}^m \sign(y_i - h_i(x)) \cdot (-e^{-\ip{a_i, x}_+}) a_i \1\set{\ip{a_i, x} \geq 0}, \label{eq:subgradient} \\
		\text{where} \quad \sign(x) & := \begin{cases}
			                                 1,       & x > 0, \\
			                                 -1,      & x < 0, \\
			                                 [-1, 1], & x = 0,
		                                 \end{cases}, \quad \text{and} \quad
		\1\set{x \geq 0} := \begin{cases}
			                    1,     & x > 0  \\
			                    0,     & x < 0, \\
			                    [0, 1] & x = 0.
		                    \end{cases}
		\label{eq:indicator-sign-defn}
	\end{align}
\end{subequations}

\begin{remark}
	\label{remark:subgradient-choice}
	Our implementation uses $\sign(0) = 0$ and $\1\set{x \geq 0} = 1$ if $x = 0$ in~\eqref{eq:indicator-sign-defn}.
\end{remark}

Much of our analysis will focus on showing that the loss $f$ in~\eqref{eq:nonsmooth-loss}
is ``well-conditioned'' in a neighborhood of the solution $\sol$, a consequence of
the following two regularity properties:
\begin{description}
	\item[(\textbf{Lipschitz continuity})] There is $\lip > 0$ such that for any $x, \bar{x} \in \Rbb^d$, the loss $f$ satisfies
	      \begin{equation}
		      \abs{f(x) - f(\bar{x})} \leq \lip \norm{x - \bar{x}}.
		      \label{eq:lipschitz-continuity}
	      \end{equation}
	\item[(\textbf{Sharp growth})] There is $\mu > 0$ such that for any $x$ near $\sol$, the loss $f$ satisfies
	      \begin{equation}
		      f(x) - f(\sol) \geq \mu \norm{x - \sol}.
		      \label{eq:sharp-growth}
	      \end{equation}
\end{description}
It is well known (see, e.g.,~\citet{Polyak69,Goffin77}) that classical
subgradient methods converge linearly for \emph{convex} functions satisfying
\eqref{eq:lipschitz-continuity}--\eqref{eq:sharp-growth},
with rate
\begin{equation}
	\norm{x_{k} - \sol}^2 \leq \left(1 - \frac{1}{\condnum^2}\right)^k
	\norm{x_0 - \sol}^2, \quad
	\condnum := \frac{\lip}{\mu}.
	\label{eq:classical-polyak-rate}
\end{equation}
In~\eqref{eq:classical-polyak-rate}, the \textbf{condition number} $\condnum$ generalizes the classical notion of
conditioning from smooth optimization to nonsmooth programs. Unfortunately,
the loss in~\eqref{eq:nonsmooth-loss} is nonconvex; however,
we show that it satisfies a key ``aiming'' condition,
postulating that subgradients --- potential search directions for~\cref{alg:polyaksgm} --- point
to the direction of the solution $\sol$.
\begin{equation}
	(\texttt{Aiming}):
	\quad
	\min_{v \in \partial f(x)} \ip{v, x - \sol} \geq \mu \norm{x - \sol}, \;\;
	\text{for all $x \in \cB(\bm{0}; 3 \norm{\sol}) \setD \set{0}$}.
	\label{eq:aiming-intro-statement}
\end{equation}
The aiming inequality~\eqref{eq:aiming-intro-statement} serves
two purposes: first, it is a key stepping stone to establishing the sharp growth
property~\eqref{eq:sharp-growth} in a neighborhood of $\sol$. Moreover,
it ensures that~\cref{alg:polyaksgm} (with
sufficiently small scaling $\eta$) decreases the distance to $\sol$ monotonically:
\[
	\norm{x_{k+1} - \sol}^2 \leq \norm{x_{k} - \sol}^2 -
	\frac{\eta \mu f(x_k)}{\norm{v_k}^2} \norm{x_k - \sol} < \norm{x_k - \sol}^2.
\]
The focal point of our analysis is establishing~\eqref{eq:aiming-intro-statement}
and showing that the moduli of sharp growth and Lipschitz continuity, $\mu$ and $\lip$,
do not scale with the dimension $d$. Indeed, we prove that
\begin{equation}
	\mu = \Omega\left(\frac{1}{\norm{\sol}^2}\right) \;\;
	\text{and} \;\; \lip = O(1),
	\quad \text{as long as $m = \Omega(d \norm{\sol}^4)$},
\end{equation}
with high probability when the measurement vectors $a_i$ are drawn from a
standard Gaussian distribution; a precise statement can be found in~\cref{theorem:main}.
This establishes that the condition number of our problem satisfies
\[
	\kappa = O(\norm{\sol}^2),
\]
i.e., it grows polynomially with the norm of the solution $\sol$. As
dicussed previously (and illustrated in~\cref{table:comparison}), this
yields an exponential improvement over the sample and iteration complexity
of the method in~\cite{FVW+23}, which minimizes a smooth penalty using gradient
descent. Though at first surprising, the advantage of nonsmooth formulations
in signal recovery problems is well-documented in prior work~\cite{LZMV20,CCD+21}
showing that the $\ell_1$ penalty is both better-conditioned and
more robust to measurement noise than the squared $\ell_2$ loss. Finally,
we note that our convergence result can be extended to structured
signal recovery by incorporating appropriate convex constraints as in
\citet[Theorem 2]{FVW+23}. Since the proof techniques are similar, we do not pursue this generalization in
our paper.

\paragraph{Paper outline.}
In the remainder of this section, we discuss related work (\cref{sec:related-work}) and
review standard notation and constructions from nonsmooth analysis and high-dimensional probability used throughout the
paper (\cref{sec:notation}).
In~\cref{sec:regularity-properties}, we prove
that the loss function~\eqref{eq:nonsmooth-loss} satisfies certain
regularity properties, treating the initialization separately in~\cref{sec:initialization}.
In~\cref{sec:convergence} we analyze the convergence
of~\cref{alg:polyaksgm} and propose
a variant that adapts to the unknown parameter $\norm{\sol}$ in~\cref{sec:convergence-adaptive}.
\Cref{sec:experiments}  presents
a set of numerical experiments that demonstrate the behavior of our method and its
competitiveness with the approach of~\cite{FVW+23}. We conclude with a discussion
of the limitations of our approach and potential extensions of our work in~\cref{sec:discussion}.
Proofs of all technical results are deferred to~\cref{appendix:sec:missing-proofs}.

\subsection{Related work}
\label{sec:related-work}

Several algorithms for CT image reconstruction, including the standard FBP method,
can be viewed as discrete approximations of analytical inversion formulas and are
relatively inexpensive to implement, in contrast to iterative reconstruction (IR) methods.
While there are reasons beyond computational complexity impeding the integration of
iterative methods in CT hardware, as discussed in the survey of~\citet{PSV09}, advances
in computing and algorithm designed have led to renewed interest in such methods~\cite{BKK12,GSM+15}.
While a comprehensive overview of IR methods is beyond the scope of this article,
and can be found in surveys such as~\cite{BKK12,WdJL+13},
several of these methods are motivated by advances in numerical optimization: examples
include the classical algebraic reconstruction technique (ART)~\cite{Gordon74},
which employs the Kaczmarz algorithm for iteratively solving linear systems~\cite{Kaczmarz93};
the SAGE method of~\citet{FH94}, which uses alternating minimization to accelerate
the expectation-maximization (EM) method;
the ASD-POCS method of~\citet{SP08}, inspired by total-variation (TV) regularization
for image recovery~\cite{ROF92}; and the nonconvex ADMM approach of~\citet{BS24}.
Few of these works provide estimates on the sample and computational efficiency
of the proposed methods and, when they do, %
the estimates are typically not adapted to CT problems.
Finally, note that commercial iterative methods for CT imaging also incorporate
proprietary knowledge, such as machine geometry and detector characteristics,
to improve algorithm design~\cite{WdJL+13}.

Beyond computed tomography, several works design and analyze first-order methods
for signal recovery in other settings; this includes works that study the
sample and computational complexity of recovering a signal from magnitude-only or quadratic measurements~\cite{Soltanolkotabi2019} (also known as \emph{phase retrieval})
and measurements produced by piecewise nonlinearities such as ReLUs~\cite{Soltanolkotabi2017,FCG20},
as well as recovering low-rank matrices using the Burer-Monteiro factorization~\cite{MWCC20,CCD+21}.
Other works include~\cite{MBM18}, who develop a framework for transfering guarantees
from population to empirical risk for smooth loss functions and~\cite{CPT23}, which develops
nonasymptotic predictions for the trajectory of certain first-order methods using Gaussian
data. To the best of our knowledge, the work of~\citet{FVW+23} is the first
to study the computational and sample complexity of signal recovery under the
particular measurement model used in computed tomography.

\subsection{Notation and basic constructions}
\label{sec:notation}
We write $\ip{X, Y} := \mathsf{Tr}(X^{\T} Y)$ for the Euclidean inner product and
$\norm{X} = \sqrt{\ip{X, X}}$ for the induced norm. We denote the unit sphere in
$d$ dimensions by $\Sbb^{d-1}$ and the Euclidean ball centered at $\bar{x}$
and radius $r$ by $\cB(\bar{x}; r)$. When $A \in \Rbb^{m \times d}$, we will write
$\frobnorm{A} = \sqrt{\ip{A, A}}$ for its \emph{Frobenius} norm and
$\opnorm{A} := \sup_{x \in \Sbb^{d-1}} \norm{Ax}$ for its \emph{spectral}
norm.
Some of our guarantees depend on the
complementary Gaussian error function, $\erfc(t)$, defined by
\begin{equation}
	\erfc(t) = \frac{2}{\sqrt{\pi}} \int_{t}^{\infty} \exp(-u^2) \dd{u}.
	\label{eq:erfc-defn}
\end{equation}
Finally, will write $A \lesssim B$ to indicate that there exists a
constant $c > 0$ such that $A \leq cB$; the precise value
of $c$ may change from occurence to occurence.

\paragraph{Orlicz norms.} We write $\norm{X}_{\psi_q}$ for
the \emph{$q$-Orlicz norm}~\cite{Orlicz32} of a random variable $X$:
\[
	\norm{X}_{\psi_q} := \inf\set{t > 0 \mmid \expec{\expfun{\left(\frac{\abs{X}}{t}\right)^q}} \leq 2} \in [0, \infty].
\]
Any $X$ with $\norm{X}_{\psi_1} < \infty$ is called \emph{subexponential}; any $X$ with
$\norm{X}_{\psi_2} < \infty$ is called \emph{subgaussian}.

\paragraph{Nonsmooth analysis.} Consider a locally Lipschitz function $f: \Rbb^{d} \to \Rbb$ and
a point $x$. The \emph{Clarke subdifferential} of $f$ at $x$~\cite{Clarke1975}, denoted by $\partial f(x)$,
is the convex hull of limits of gradients evaluated at nearby points:
\begin{equation}
	\tag{\texttt{Clarke}}
	\partial f(x) := \convhull\set{
		\lim_{i \to \infty} \grad f(x_i) \mid x_{i} \overset{\Omega}{\to} x
	},
	\label{eq:clarke-subdifferential}
\end{equation}
where $\Omega \subset \Rbb^d$ is the set of points at which $f$ is differentiable. In particular,
if $f$ is $\lip$-Lipschitz on an open set $U$, then all $x \in U$ and $v \in \partial f(x)$
satisfy $\norm{v} \leq \lip$ (and vice-versa).

\section{Regularity properties of the loss function}
\label{sec:regularity-properties}
In this section, we show that the loss function satisfies certain
regularity properties, including Lipschitz continuity (\cref{sec:lipschitz-continuity})
and sharp growth (\cref{sec:sharpness}). We address the
initial point $x_{0} = \bm{0}$ separately, establishing that the
first iterate $x_1$ is positively correlated with $\sol$ and bounded away
from $0$ to facilitate the use of the analytical framework outlined in~\cref{sec:method-overview}.

\subsection{Properties at initialization}
\label{sec:initialization}
We show that when $x_{0} = \bm{0}$, the first iterate attains nontrivial
correlation with the solution $\sol$. To do so, we analyze the different
components making up the first subgradient step.

\begin{lemma}
	\label{lemma:loss-concentration-nosimpl}
	The loss function satisfies
	\begin{equation}
		\prob{\abs{f(0) - \frac{1}{2} \left(1 - \expfun{\frac{\norm{\sol}^2}{2}}
				\erfc\left(\frac{\norm{\sol}}{\sqrt{2}}\right) \right)}
			\geq \norm{\sol}\sqrt{\frac{d}{m}}
		} \leq 2\expfun{-d}.
		\label{eq:loss-concentration-nosimpl}
	\end{equation}
\end{lemma}
An immediate consequence of~\cref{lemma:loss-concentration-nosimpl} is that
$f(0)$ is of size $\Omega(1)$ when $\norm{\sol} = \Omega(1)$.
\begin{corollary}
	\label{corollary:loss-concentration-simpl}
	For any $\sol$ such that $\norm{\sol} \geq 1$ and
	$m \gtrsim d \cdot \norm{\sol}^2$, it holds that
	\begin{equation}
		\prob{f(0) \leq \frac{1}{5}} \leq \expfun{-\frac{d}{2}},
	\end{equation}
	where $c > 0$ is a universal constant.
\end{corollary}
The next Lemma shows that a subgradient at $0$ is bounded and correlated with $-\sol$,
leading to sufficient decrease in distance in the first step of~\cref{alg:polyaksgm}.
\begin{lemma}
	\label{lemma:subgradient-correlation-at-0}
	Consider the following subgradient at $0$:
	\begin{equation}
		v_0 := -\frac{1}{m} \sum_{i = 1}^m a_i \1\set{\ip{a_i, \sol} > 0} \in \partial f(0).
	\end{equation}
	Then with probability at least $1 - 2e^{-d}$, it holds that
	\begin{subequations}
		\begin{align}
			-\ip{v_0, \sol}   & \geq \norm{\sol} \left(\sqrt{\frac{1}{2 \pi}} - \sqrt\frac{2d}{m}\right), \label{eq:v0-correlation} \\
			\norm{v_0}        & \leq 1 + 2 \sqrt{\frac{d}{m}}, \label{eq:v0-x-norm}                                                 \\
			\norm{x_1 - \sol} & \leq
			\norm{\sol} \left(
			1 - \frac{\eta}{10 \norm{\sol} \sqrt{\pi}}
			\right)^{\sfrac{1}{2}},
			\label{eq:initial-distance}
		\end{align}
	\end{subequations}
	as long as $m \gtrsim d$ and $\eta \leq \frac{1}{2}$.
\end{lemma}

\subsection{Lipschitz continuity}
\label{sec:lipschitz-continuity}
We prove that the loss function is Lipschitz-continuous; note
that the Lipschitz modulus is dimension-independent when $m \gtrsim d$.

\begin{proposition}[Lipschitz continuity]
	\label{prop:lipschitz}
	The loss function is $\lip$-Lipschitz continuous, where $\lip := 1 + 2 \sqrt{\frac{d}{m}}$,
	with probability at least $1 - \expfun{-d}$:
	\begin{equation}
		\abs{f(x) - f(\bar{x})} \leq \lip \norm{x - \bar{x}}, \quad
		\text{for all $x, \bar{x} \in \Rbb^d$.}
	\end{equation}
\end{proposition}
This immediately implies the following upper bound on subgradient norms:
\begin{corollary}
	With probability at least $1 - \expfun{-d}$, we have that
	\begin{equation}
		\sup_{x \in \Rbb^d} \max_{v \in \partial f(x)} \norm{v} \leq 1 + 2 \sqrt{\frac{d}{m}}.
		\label{eq:subgradient-norm-bound}
	\end{equation}
\end{corollary}

\subsection{Sharpness}
\label{sec:sharpness}
In this section, we establish that the loss function $f$ grows \emph{sharply} away from its minimizer:
\[
	f(x) - f(\sol) \geq \mu \norm{x - \sol},
	\quad \text{for all $x$ ``near'' $\sol$,}
\]
where $\mu = \frac{1}{4 \sqrt{\pi} (1 + 9 \pi \norm{\sol}^2)}$. For technical reasons, we will prove this claim
for all $x \in \cB(\bm{0}; 2 \norm{\sol})$, as we will later show that all iterates of
\cref{alg:polyaksgm} initialized at $\bm{0}$ remain within that ball. Key
to establishing the above claim is the following ``aiming'' inequality:
\begin{align}
	\tag{\texttt{Aiming}}
	\min_{v \in \partial f(x)} \ip{v, x - \sol} \geq \mu \norm{x - \sol}, \quad
	\text{for all $x \in \cB(\bm{0}; 3 \norm{\sol}) \setD \set{0}$}.
	\label{eq:aiming-abstract}
\end{align}
The reason~\eqref{eq:aiming-abstract} implies sharpness for $f$ is captured in
the following technical result:
\begin{lemma}[Solvability lemma]
	\label{lemma:sharpness-from-aiming}
	Suppose $f$ is Lipschitz on $\cB(\bm{0}; 3 \norm{\sol})$ and that
	the~\eqref{eq:aiming-abstract} inequality holds. Then we have that
	\begin{equation}
		f(x) - f_{\star} \geq \mu
		\min\set{\norm{x - x_{\star}}, \norm{x}}, \;\;
		\text{for all $x \in \cB(\sol; \norm{\sol})$.}
		\label{eq:sharpness-from-aiming}
	\end{equation}
\end{lemma}

\begin{remark}
	\label{remark:solvability-and-sharpness}
	While the lower bound furnished by~\cref{lemma:sharpness-from-aiming} is weaker
	than the sharp growth inequality~\eqref{eq:sharp-growth}, our convergence analysis
	shows that the norm of the iterates produced by~\cref{alg:polyaksgm}
	stays bounded away from $0$ and, after an initial ``burn-in'' phase, surpasses the distance to $\sol$.
	Consequently,~\eqref{eq:sharpness-from-aiming} eventually reduces to the sharp
	growth property.
\end{remark}
We now turn to the proof of the aiming inequality.
\subsubsection{Proof of the aiming inequality}
To establish~\eqref{eq:aiming-abstract}, we first derive a convenient expression for the
inner product.

\begin{lemma}[Subdifferential inner product]
	\label{lemma:aiming-inner-product}
	For any point $x \in \Rbb^d$, we have that
	\begin{equation}
		\ip{\partial f(x), x - \sol} =
		\begin{aligned}[t]
			 & \frac{1}{m} \sum_{i = 1}^m e^{-\ip{a_i, x}} \abs{\ip{a_i, x - \sol}} \1\set{\ip{a_i, x} > 0}                    \\
			 & + \frac{1}{m} \sum_{i=1}^m \1\set{\ip{a_i, x} = 0} \left(\ip{a_i, \sol}_+ + \sign(0) \ip{a_i, \sol}_{-}\right).
		\end{aligned}
		\label{eq:aiming-inner-product}
	\end{equation}
\end{lemma}
With the help of~\cref{lemma:aiming-inner-product}, we can establish
the desired inequality in expectation.
\begin{lemma}[Aiming in expectation]
	\label{lemma:aiming-lb}
	For any $x \in \cB(\bm{0}; 3 \norm{\sol}) \setD \set{0}$, it follows that
	\begin{equation}
		\mathbb{E}[\ip{\partial f(x), x - \sol}]
		\geq \frac{\norm{x - \sol}}{\sqrt{8 \pi}\left(1 + 9 \pi \norm{\sol}^2\right)}.
		\label{eq:aiming-lb}
	\end{equation}
\end{lemma}
Finally, we establish a uniform deviation inequality around the bound furnished
by~\cref{lemma:aiming-lb}.
\begin{proposition}
	\label{prop:aiming-unif-lb}
	The following holds with probability at least $1 - 2e^{-d}$:
	\begin{equation}
		\min_{x \in \cB(\bm{0}; 3 \norm{\sol}) \setminus \set{\sol, \bm{0}}}
		\frac{\ip{\partial f(x), x - \sol}}{\norm{x - \sol}}
		\geq \frac{1}{4 \sqrt{\pi}\left(1 + 9 \pi \norm{\sol}^2\right)},
		\label{eq:aiming-unif-lb}
	\end{equation}
	as long as $m \gtrsim d \cdot \norm{\sol}^4$.
\end{proposition}

\section{Convergence analysis}
\label{sec:convergence}
In this section, we analyze the convergence of~\cref{alg:polyaksgm}.
First, we denote
\begin{equation}
	\overline{\rho} := \eta \left(\frac{\mu}{\lip}\right)^2 \frac{1}{40 \sqrt{\pi} \norm{\sol}}, \quad
	\rho := \eta \left(\frac{\mu}{\lip}\right)^2, \quad
	\text{and} \quad
	T_{0} := \ceil{\frac{\log(2)}{\rhoslow}}
	\label{eq:convergence-shorthand}
\end{equation}
for brevity. A quick discussion about their roles in the convergence analysis is in order:
\begin{itemize}
	\item $\overline{\rho}$ is the contraction factor achieved while $\norm{x_t - \sol} \geq \frac{1}{2} \norm{\sol}$.
	\item $\rho > \overline{\rho}$ is the contraction factor achieved after $\norm{x_t - \sol} < \frac{1}{2} \norm{\sol}$.
	\item $T_0$ is an upper bound on the number of iterations elapsed until $\norm{x_t - \sol} < \frac{1}{2} \norm{\sol}$.
\end{itemize}
As the above quantities might suggest, our convergence analysis is split into two phases. Crucially,
\cref{alg:polyaksgm} \textbf{does not}
employ a different stepsize for the two phases: the distinction is only for theoretical purposes
and -- as demonstrated in~\cref{sec:experiments} -- does not affect the practical behavior of
the method. For that reason, we believe it is just an artifact of our analysis.

We now turn to the analysis of the algorithm. We define the following events:
\begin{subequations}
	\begin{align}
		\cA_{\slow}(t) & := \set{\norm{x_{t+1} - \sol}^2 \leq \left(1 - \overline{\rho} \right) \norm{x_t - \sol}^2}, \\
		\cA_{\fast}(t) & := \set{\norm{x_{t+1} - \sol}^2 \leq (1 - \rho) \norm{x_t - \sol}^2},                        \\
		\cB_{\slow}(t) & := \set{\frac{1}{40\sqrt{\pi}} \leq \norm{x_t} \leq 2 \norm{\sol}},                          \\
		\cB_{\fast}(t) & := \set{\frac{1}{2} \norm{\sol} \leq \norm{x_t} \leq 2 \norm{\sol}}.
	\end{align}
\end{subequations}
We also recall that $x_{0} = \bm{0}$ and $f_{\star} = 0$ throughout.
Our analysis is inductive: initially, we note that~\cref{lemma:subgradient-correlation-at-0} implies the following:
\begin{equation}
	\norm{x_1 - \sol}^2 \leq \left(1 - \frac{1}{20 \sqrt{\pi} \norm{\sol}}\right) \norm{\sol}^2.
	\label{eq:x1-distance-reminder}
\end{equation}
We now turn to a sequence of supporting Lemmas. The first one shows that the norms
of the iterates remain bounded while the algorithm is in its ``slow'' phase.
\begin{lemma}
	\label{lemma:Aslow-implies-Bslow}
	We have that $\set{\cA_{\slow}(j)}_{j \leq t} \implies \cB_{\slow}(t+1)$.
\end{lemma}

The next Lemma shows that the algorithm continues making progress (at a rate
depending on $\rhoslow$) while
the iterates remain within the tube $\cB(\bm{0}; 2 \norm{\sol}) \setD \cB(\bm{0}; \sfrac{1}{40 \sqrt{\pi}})$.
\begin{lemma}
	\label{lemma:allslow-implies-Aslow}
	We have that $\set{\cA_{\slow}(j)_{j < t}, \cB_{\slow}(t)} \implies \cA_{\slow}(t)$ for any $\eta \leq \frac{\mu}{\mathsf{L}}$
	and $t \geq 1$.
\end{lemma}

The forthcoming Lemmas describe the behavior of the algorithm once iterates are
in the ball $\cB(\sol; \frac{1}{2} \norm{\sol})$. \Cref{lemma:Aslow-until-T0-implies-Bfast}
shows that the algorithm indeed enters this region after sufficient progress;
the remaining Lemmas mirror~\cref{lemma:Aslow-implies-Bslow,lemma:allslow-implies-Aslow}.
\begin{lemma}
	\label{lemma:Aslow-until-T0-implies-Bfast}
	We have $\set{\cA_{\slow}(j)}_{j < T_0} \implies \cB_{\fast}(T_0)$.
\end{lemma}

\begin{lemma}
	\label{lemma:Allfast-implies-Afast}
	We have that $\set{\cA_{\fast}(j)_{T_0 \leq j < t}, \cB_{\fast}(t)} \implies \cA_{\fast}(t)$
	for any $\eta \leq \frac{\mu}{\lip}$ and $t \geq T_0$.
\end{lemma}

\begin{lemma}
	\label{lemma:Afast-implies-Bfast}
	We have $\set{\cA_{\fast}(j)}_{T_0 \leq j \leq t} \implies \cB_{\fast}(t+1)$.
\end{lemma}

We are now ready to state our main convergence guarantee, which follows trivially
from~\cref{lemma:Aslow-implies-Bslow,lemma:allslow-implies-Aslow,lemma:Aslow-until-T0-implies-Bfast,lemma:Allfast-implies-Afast,lemma:Afast-implies-Bfast}.
To that end, let
\begin{equation}
	\condnum := 8 \sqrt{\pi}(1 + 9 \pi \norm{\sol}^2) \geq \frac{\lip}{\mu}
	\label{eq:condition-number}
\end{equation}
denote a bound on the condition number of the optimization problem. When $\eta = \frac{\mu}{\lip}$,
we have
\begin{equation}
	\rho = \eta \left(\frac{\mu}{\lip}\right)^2 \leq \kappa^{-3}, \;\;
	\overline{\rho} = \rho \cdot \frac{1}{40 \sqrt{\pi} \norm{\sol}} \geq
	\frac{1}{5} \rho \kappa^{-\sfrac{1}{2}}.
	\label{eq:rhofast-vs-rhoslow-vs-condnum}
\end{equation}
\begin{theorem}[Linear convergence of~\texttt{PolyakSGM}]
	\label{theorem:main}
    Let~\cref{assm:measurement-model} hold and suppose that the number of measurements $m \gtrsim d \norm{\sol}^4$, for $\norm{\sol} \geq 1$.
	Then~\cref{alg:polyaksgm} with inputs $x_0 = \bm{0}$, $\eta \leq \frac{1}{\kappa}$, where $\kappa$
	is defined in~\eqref{eq:condition-number}, and $f_{\star} = 0$
	produces a set of iterates $\set{x_i}_{i = 1, \dots}$ converging to $\sol$ at a linear rate:
	\begin{equation}
		\norm{x_{i+1} - \sol}^2 \leq \begin{cases}
			(1 - \frac{1}{5} \eta \kappa^{-\frac{5}{2}} ) \norm{x_{i} - \sol}^2, & i < T_0,    \\
			(1 - \eta \kappa^{-2}) \norm{x_{i} - \sol}^2,                        & i \geq T_0.
		\end{cases}
	\end{equation}
	In particular, using the largest admissible stepsize $\eta = \frac{1}{\kappa}$, \cref{alg:polyaksgm} requires at most
	\[
		\mathsf{T}(\varepsilon) \leq \ceil{7 \kappa^{\frac{7}{2}}} +
		\ceil{2 \kappa^3 \log\left(\frac{\kappa}{\varepsilon}\right)}
		\quad \text{subgradient evaluations}
	\]
	to achieve estimation error $\norm{x_t - \sol} \leq \varepsilon$.
\end{theorem}

\subsection{Adaptation to $\norm{x_{\star}}$}
\label{sec:convergence-adaptive}
In~\cref{theorem:main}, the optimal $\eta$ depends on the norm of the unknown signal $x_{\star}$,
which affects the condition number $\kappa$.
Here, we describe a data-driven approach to estimate that quantity at a constant expense.
Recall the definition of $\condnum$ in~\eqref{eq:condition-number};
the optimal $\eta_{\star}$ in~\cref{theorem:main} is equal to $\frac{1}{\condnum}$.
For any $\condest \geq \condnum$ and $\eta := \condest^{-1}$, we have (during the second stage of the analysis):
\begin{align}
	\norm{x_{t} - \sol}^2 & \leq \left(1 - \eta \condnum^{-2}\right)^{t} \norm{x_0 - \sol}^2             \notag   \\
	                      & = \left(1 - {\condest^{-1} \cdot \condnum^{-2}}\right)^{t} \norm{x_0 - \sol}^2 \notag \\
	                      & \leq \left(1 - \condest^{-3}\right)^t \norm{x_0 - \sol}^2.
	\label{eq:surrogate-bound}
\end{align}
In particular, a standard scheme to adapt $\eta$ to the unknown $\norm{x_{\star}}$ (given a budget of $T$ iterations)
would proceed by setting $\eta_{i} = {\condest_i^{-1}}$, with $\condest_1 = 1$, and doubling
$\condest_{t+1} = 2 \condest_t$ until the estimate~\eqref{eq:surrogate-bound} holds. To bypass
the obvious issue -- namely, that the distances $\norm{x_j - \sol}$ are not available in practice --
we use the following relationship (cf.~\cref{prop:lipschitz,prop:aiming-unif-lb}):
\[
	\frac{1}{\mathsf{L}} f(x_{j}) \leq \norm{x_j - \sol} \leq \frac{1}{\mu} f(x_j),
\]
which leads to the following \emph{computable} upper bound (valid for any $\condest \geq \condnum$):
\begin{equation}
	f(x_t) \leq \condnum \cdot \left(1 - \condest^{-3}\right)^{t/2} f(x_0) \leq
	\condest \left(1 - \condest^{-3}\right)^{t/2} f(x_0).
	\label{eq:surrogate-bound-computable}
\end{equation}
With~\eqref{eq:surrogate-bound-computable} at hand, we describe our adaptation strategy below.
\begin{algorithm}
	\caption{\texttt{AdPolyakSGM}}
	\begin{algorithmic}
		\State \textbf{Input}: $x_0 \in \mathbb{R}^d$, target accuracy $\varepsilon$.
		\State Set $i = 0$, $\condest_{i} = 1$.
		\Repeat
		\State $T_{i} := \ceil{7 \condest_i^{\frac{7}{2}}} + \ceil{2 \cdot \condest_{i}^3 \log \left(\frac{\kappa_i}{\varepsilon}\right)}$ \label{op:iterations}
		\State $x_i := \texttt{PolyakSGM}\left(x_0, \condest_i^{-1}, T_{i}, f_{\star} \equiv 0\right)$
		\Comment{Algorithm~\ref{alg:polyaksgm}}
		\State $\condest_{i+1} = 2 \condest_{i}$; $i \gets i + 1$.
		\Until{$f(x_i) \leq \varepsilon \cdot f(x_0)$}
	\end{algorithmic}
	\label{alg:adpolyak}
\end{algorithm}

\begin{proposition}
	\label{prop:adpolyak-termination}
    Let~\cref{assm:measurement-model} hold and suppose that the number of measurements
    $m \gtrsim d \norm{\sol}^4$, for $\norm{\sol} \geq 1$.
	Then~\cref{alg:adpolyak} with $x_0 = \bm{0}$ terminates after at most $\log_{2}(\condnum)$ invocations
	of~\cref{alg:polyaksgm} and returns
	a point $\widehat{x}$ satisfying $f(\widehat{x}) \leq \varepsilon f(x_0)$.
	Moreover, it uses at most
	\[
		8 \condnum^{\frac{7}{2}} + 3 \condnum^3 \cdot \log\left(\frac{\condnum}{\varepsilon}\right)
		\quad \text{subgradient evaluations}.
	\]
\end{proposition}

\begin{remark}
	Note that, since $\condnum = \Theta(\norm{\sol}^2)$, the cost of~\cref{alg:adpolyak} (as
	measured by the total number of subgradient evaluations) is just a constant factor
	worse compared to the cost of running~\cref{alg:polyaksgm} with the optimal stepsize
	$\eta_{\star} = \condnum^{-1}$.
\end{remark}

\section{Numerical experiments}
\label{sec:experiments}
In this section, we evaluate our proposed algorithm in experiments using real
and synthetic data. Our primary goal is to compare the performance of our
method with the gradient-based algorithm introduced in~\cite{FVW+23}; that
work minimizes a squared version of the loss using gradient descent with
constant stepsize:
\begin{subequations}
	\begin{align}
		\cL(x)                & = \frac{1}{2m} \sum_{i = 1}^m \left(1 - \exp(-\ip{a_i, x}_+) - y_i\right)^2, \\
		x_{k+1}^{\texttt{GD}} & = x_{k}^{\texttt{GD}} - \eta_k \cdot \grad \cL(x_k^{\texttt{GD}}),
		\qquad \eta_k = \begin{cases}
			                c_0 \cdot e^{-5 \norm{\sol}},                                                          & k \geq 1, \\
			                \frac{4 e^{-\frac{\norm{\sol}^2}{2}}}{\erfc\left(\frac{\norm{\sol}}{\sqrt{2}}\right)}, & k = 0,
		                \end{cases}
		\label{eq:gd-baseline}
	\end{align}
\end{subequations}
where $c_0$ is a constant and $x_{0}^{\texttt{GD}} = \bm{0}_{d}$ is used to initialize the method.

Our synthetic experiments use two different models for measurement vectors $a_i$:
\begin{description}
	\item[(Gaussian ensemble)] Every $a_i$ satisfies $a_i \sim \cN(0, I_{d})$. This is the setting
	      covered by our theoretical analysis (\cref{theorem:main}).
	      \label{item:model:gaussian}
	\item[(\textsf{RWHT} ensemble)] Every $a_i$ is a row of a randomized Walsh-Hadamard
	      transform with oversampling. Formally, given an oversampling factor $\ell = \frac{m}{d} \in \Nbb$,
	      we have
	      \[
		      A = \begin{bmatrix}
			      a_1^{\T} \\
			      a_2^{\T} \\
			      \vdots   \\
			      a_{m}^{\T}
		      \end{bmatrix}
		      = \begin{bmatrix}
			      H_{d} D_1 \\
			      H_{d} D_2 \\
			      \vdots    \\
			      H_{d} D_{\ell}
		      \end{bmatrix}, \quad
		      \text{where} \quad
		      D_{j} = \mathbf{diag}(\xi_1^{(j)}, \dots, \xi_d^{(j)}), \;\;
		      \xi_i^{(j)} \iid \mathrm{Unif}(\pm 1),
	      \]
	      where $H_{d}$ is the $d \times d$ matrix implementing the Walsh-Hadamard transform.
	      In order for $H_{d}$ to be well-defined, we assume that $d$ is an integer power of $2$
	      for simplicity.

	      The operations $v \mapsto Av$ and $u \mapsto A^{\T} u$ can be implemented
	      using $O(d \log d)$ \texttt{flops} and $O(m)$ memory using the Fast Walsh-Hadamard transform.
	      This allows us to test our method for values of $d$ that
	      prohibit storing a dense $m \times d$ matrix.

	      \label{item:model:rfwht}
\end{description}

\subsection{Comparison of~\texttt{PolyakSGM} and gradient descent}

\subsubsection{Sample efficiency}
We first study the sample efficiency of our method and gradient descent~\eqref{eq:gd-baseline}.
Recall that Theorem 1 in~\cite{FVW+23} provides convergence guarantees for the latter, requiring
\[
	m = \Omega\left(
	\frac{\exp(c \norm{\sol})}{\norm{\sol}^2} \cdot d
	\right) \quad \text{samples},
\]
while our~\cref{theorem:main} only requires
\[
	m = \Omega(\norm{\sol}^4 \cdot d) \quad \text{samples},
\]
an exponential improvement. To check if this sample efficiency gap manifests in practice, we generate synthetic problem instances for varying
signal scales $\norm{\sol} \in \set{1, 1.5, 2, \dots, 8}$ and oversampling factors $\frac{m}{d} \in \set{2, 4, \dots, 16}$,
using Gaussian measurements and a fixed dimension $d = 128$. We solve each instance with~\cref{alg:polyaksgm} (with fixed $\eta = 1$ across
all instances) and gradient descent (with optimized step-size $\eta$, chosen to minimize the final estimation error, for each individual instance).
We declare a solve successful if an estimate $\widehat{x}$ satisfying $\norm{\widehat{x} - \sol} \leq 10^{-5}$ is found
within $10^4$ iterations and calculate the empirical success probability over $25$
independently generated problem instances for each configuration. The results of the experiment
are depicted in~\cref{fig:phase-transition}. We observe that gradient descent experiences a
sharp cut-off in recovery probability, while \cref{alg:polyaksgm} can recover
higher-energy signals with substantially fewer measurements.

\begin{figure}[h]
	\centering
	\includegraphics[width=0.7\linewidth]{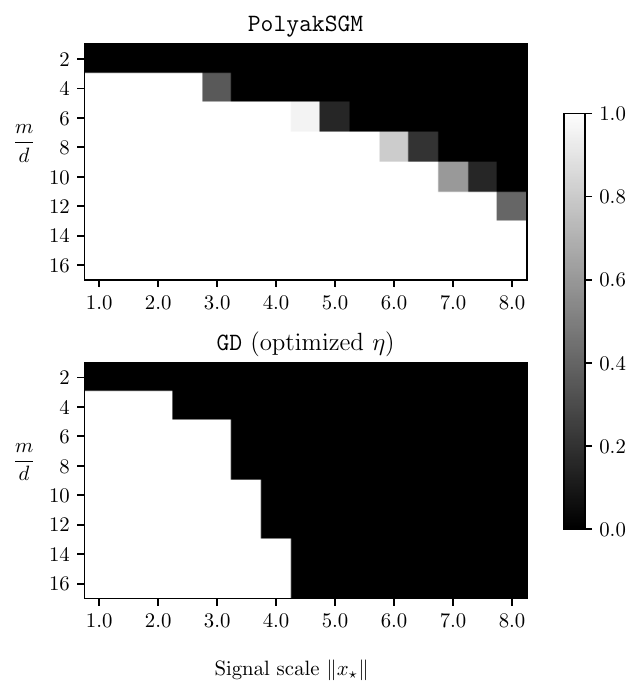}
	\caption{Empirical probability of recovering $\sol$ for various signal scales and oversampling factors
		over synthetic problems with $d = 128$ and Gaussian measurements. The empirical probability for each
		configuration is calculated over $25$ independently generated problem instances. Lighter tiles indicate higher probability
		of recovery. We compare~\cref{alg:polyaksgm} with $\eta = 1$
		against gradient descent with optimized stepsize $\eta$ using a logarithmically spaced grid.}
	\label{fig:phase-transition}
\end{figure}

\subsubsection{Optimized stepsizes}
\label{sec:polyak-vs-gd-optimized}
We now compare the convergence behavior of~\cref{alg:polyaksgm} against gradient descent~\eqref{eq:gd-baseline}
on different synthetic problems. We use~\cref{alg:polyaksgm} with pre-factor
$\eta = 1$, corresponding to the stepsize originally proposed by~\citet{Polyak69}.
To ensure a fair comparison, we experimented with different values of $c_0$
for~\eqref{eq:gd-baseline}, searching over a grid $\set{2^{j} \cdot \exp(5 \norm{\sol}) \mid -3 \leq j \leq 3}$ for each
instance. We tried different values
for the dimension $d \in \set{256, 512, 1024}$ and oversampling ratios $\frac{m}{d} \in \set{4, 8}$.
As the results in~\cref{fig:polyak-vs-gd} show, our method produces iterates
that converge to the solution $\sol$ nearly an order of magnitude faster than
the baseline method.

\begin{figure}[h]
	\centering
	\begin{subfigure}[t]{0.95\textwidth}
		\includegraphics[width=\textwidth]{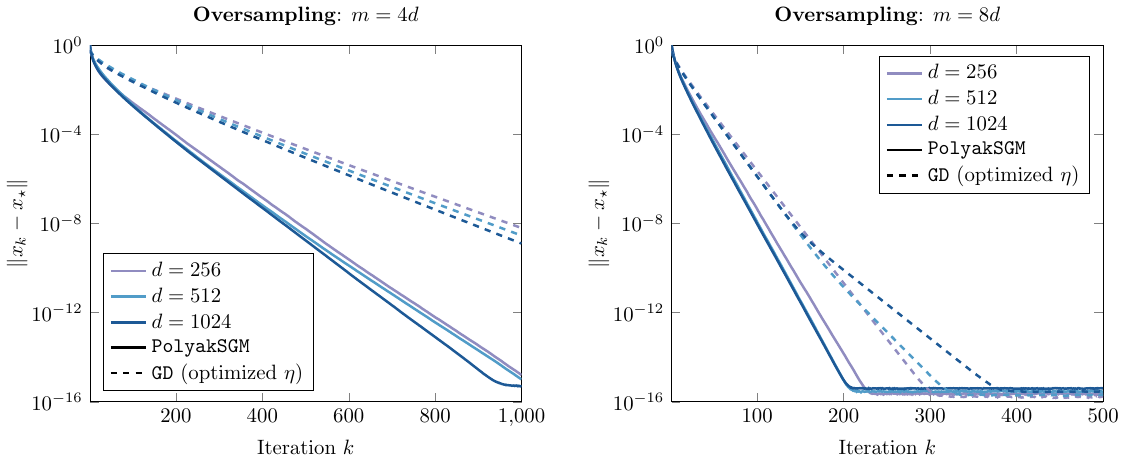}
		\caption{Performance of~\cref{alg:polyaksgm} against gradient descent with optimized stepsize $\eta$
			for measurement vectors $a_i$, $i = 1, \dots, m$ generated from a Gaussian ensemble.}
		\label{fig:polyak-vs-gd-gaussian}
	\end{subfigure} \\
	\vspace{8pt}
	\begin{subfigure}[t]{0.95\textwidth}
		\includegraphics[width=\textwidth]{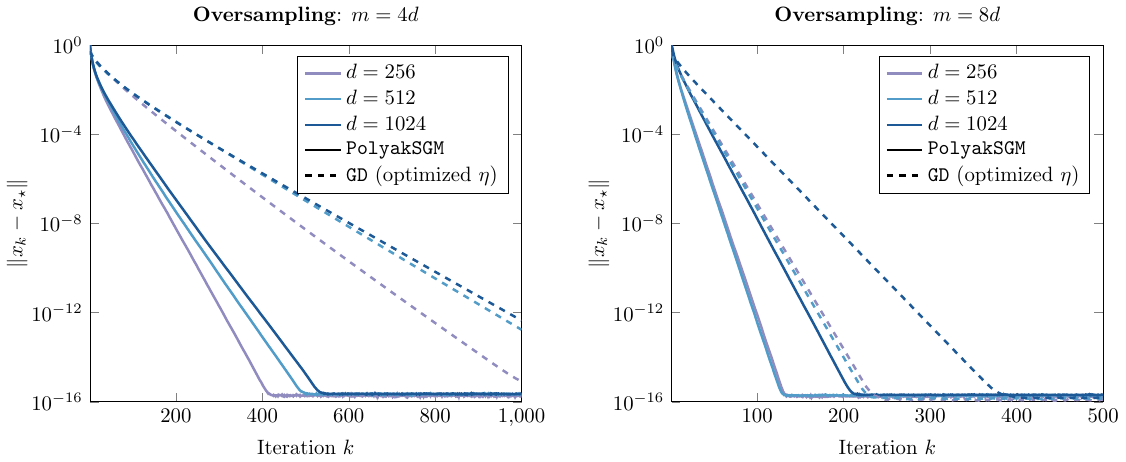}
		\caption{Performance of~\cref{alg:polyaksgm} against gradient descent with optimized stepsize $\eta$
			for measurement vectors $a_i$, $i = 1, \dots, m$ generated from a randomized Walsh-Hadamard (RWHT)
			ensemble.}
		\label{fig:polyak-vs-gd-rwfht}
	\end{subfigure}
	\vspace{4pt}
	\caption{Comparison of \texttt{PolyakSGM} (\cref{alg:polyaksgm}) using $\eta = 1$ with gradient descent~\eqref{eq:gd-baseline} using tuned stepsize $\eta$
		for synthetic problems with Gaussian and randomized Walsh-Hadamard measurements.
		The iterates produced by~\cref{alg:polyaksgm} converge to $\sol$ at moderately
		faster rates.
	}
	\label{fig:polyak-vs-gd}
\end{figure}

\subsubsection{Theoretically prescribed stepsizes}
\label{sec:polyak-vs-gd-theory}
Next, we compare~\cref{alg:polyaksgm} equipped with the stepsize prescribed by~\cref{theorem:main}:
\[
	\eta = \frac{1}{8 \sqrt{\pi} (1 + 9 \pi \norm{\sol}^2)}
\]
against gradient descent with the stepsize proposed by~\cite{FVW+23}, given in~\eqref{eq:gd-baseline}.
We posit that the primary advantage of our method stems from the milder dependence on $\norm{\sol}$
and generate instances with $\norm{\sol} \in \set{1, 2, 4}$ to
verify this hypothesis. For our experiments, we fix the number of samples to
$m = d \norm{\sol}^4$, as suggested by~\cref{theorem:main}.
The results in~\cref{fig:polyak-vs-gd-theory} confirm that the convergence rate of our
method decays more gracefully with the norm $\norm{\sol}$. At the same time, we observe
that the gradient method stagnates when $\norm{\sol} = 4$, while our method does not; this
may be due to the fact that $m = d \norm{\sol}^4$ samples are below the sample complexity
threshold from~\cite{FVW+23}. Note that the initial iterate distance in~\cref{fig:polyak-vs-gd-theory}
appears lower for the gradient method due to the use of a carefully crafted initialization~\eqref{eq:gd-baseline}.

\begin{figure}[h]
	\centering
	\includegraphics[width=\linewidth]{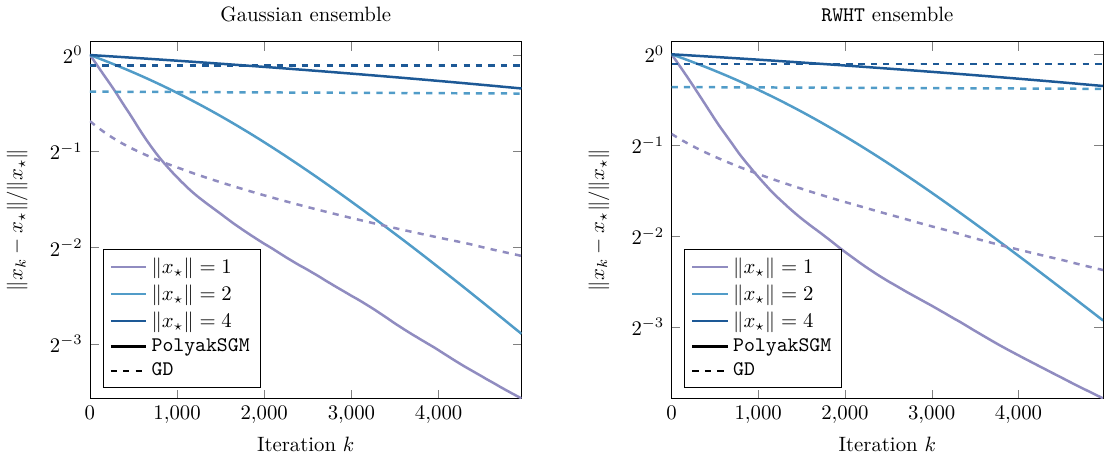}
	\caption{
		Comparison between~\cref{alg:polyaksgm} and~\eqref{eq:gd-baseline} using theoretically
		prescribed stepsizes for synthetic problem instances. In all cases,~\cref{alg:polyaksgm}
		converges faster than gradient descent.
	}
	\label{fig:polyak-vs-gd-theory}
\end{figure}

\subsubsection{Robustness to stepsize misspecification}
\label{sec:maxiter-comparison}
When deploying optimization methods, it is common to try a range of stepsizes and
use the solution achieving the lowest loss. To that end, we conduct an experiment
comparing the robustness of~\cref{alg:polyaksgm} and~\eqref{eq:gd-baseline} to different
stepsizes $\eta$. In particular, we report the number of iterations elapsed
by each method to find a solution $\widehat{x}$ with distance $\norm{\widehat{x} - \sol} \leq
	\varepsilon$, using $\varepsilon = 10^{-5}$ and a maximal budget of $10^{4}$ iterations (we terminate
any run that has not achieved the target suboptimality within the iteration limit).
We fix $d = 256$, $m = 4d$ and $\norm{\sol} = 1$ and initialize both methods at $x_{0} = \bm{0}$.
We solve 10 independently generated instances for each configuration and report the
median number of iterations $\pm$ 3 standard deviations; the results are shown
in~\cref{fig:maxiter}. While the best configuration yields only slightly faster
convergence for \texttt{PolyakSGM} compared to gradient descent, the latter method
exhibits higher variance. We also observe that the discrepancy between best-performing
and theoretically prescribed stepsizes is much starker for gradient
descent compared to \texttt{PolyakSGM}, suggesting that the convergence analysis of
gradient descent may be substantially improved.

\begin{figure}[h]
	\centering
	\begin{subfigure}[b]{0.475\linewidth}
		\includegraphics[width=\textwidth]{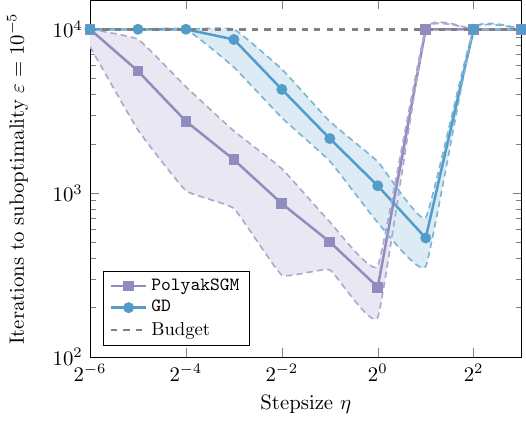}
		\caption{Gaussian measurements}
		\label{fig:maxiter:gaussian}
	\end{subfigure}%
	\hfill
	\begin{subfigure}[b]{0.475\linewidth}
		\includegraphics[width=\textwidth]{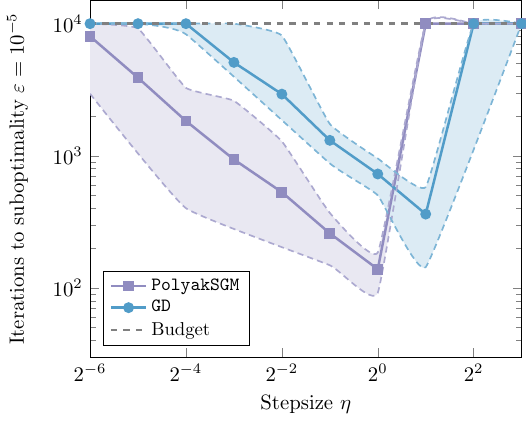}
		\caption{RWHT experiments}
		\label{fig:maxiter:rfwht}
	\end{subfigure}\\[1em]
	\caption{Number of iterations to achieve suboptimality of $\varepsilon = 10^{-5}$
		for~\cref{alg:polyaksgm} and gradient descent for different stepsizes $\eta$.
		Reporting median over 10 independent runs (solid line) $\pm$ 3 standard
		deviations (shaded area); see~\cref{sec:maxiter-comparison} for description.}
	\label{fig:maxiter}
\end{figure}

\subsection{Comparison of $\texttt{AdPolyakSGM}$ and $\texttt{PolyakSGM}$}
\label{sec:experiments:adaptive-nonadaptive}
We compare the empirical performance of \texttt{AdPolyakSGM} (\cref{alg:adpolyak}) and \texttt{PolyakSGM} (\cref{alg:polyaksgm})
on synthetic problems with measurements generated by Gaussian and RWHT ensembles. Recall that~\cref{prop:adpolyak-termination}
implies that the iteration complexity of~\cref{alg:adpolyak} is at most a
constant factor worse than that of the  ``optimal'' version. Our experiments,
shown in~\cref{fig:adaptive-vs-nonadaptive}, provide the following insights:
\begin{enumerate}
	\item The standard Polyak stepsize (corresponding to $\eta = 1$) already
	      yields good performance.
	\item The number of subgradient evaluations of~\cref{alg:adpolyak} is within a small constant
	      factor of the ``optimal'' number, as predicted by our theory.
\end{enumerate}
In particular, the first point above indicates that the standard Polyak step size
might be sufficient to guarantee convergence for the CT problem. Although beyond the scope of this
paper, a theoretical investigation of this phenomenon is an exciting direction for future work.

\begin{figure}[h]
	\centering
	\begin{subfigure}[t]{0.95\textwidth}
		\includegraphics[width=\textwidth]{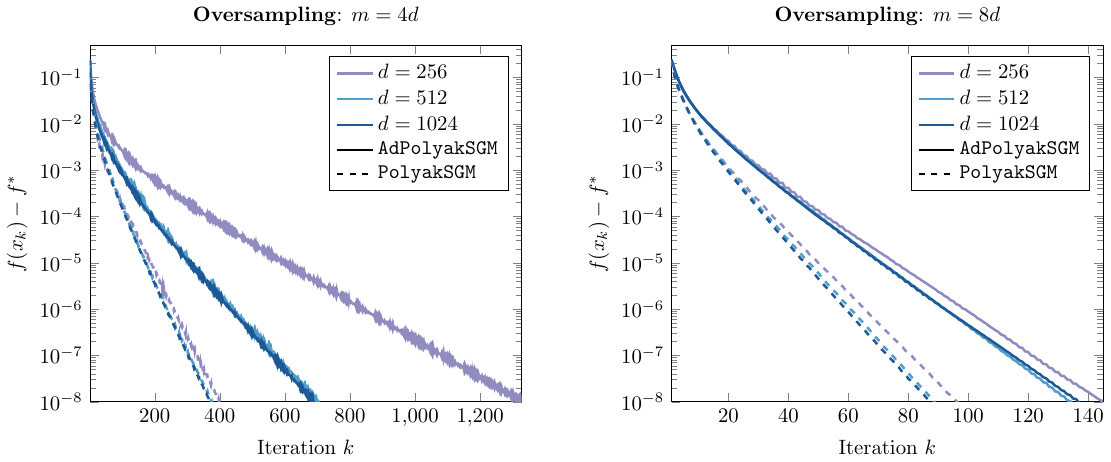}
		\caption{Performance of~\cref{alg:polyaksgm} with $\eta = 1$ against~\cref{alg:adpolyak}
			for measurement vectors $a_i$, $i = 1, \dots, m$ generated from a Gaussian ensemble.}
		\label{fig:adaptive-vs-nonadaptive-gaussian}
	\end{subfigure}\\[1em]
	\vspace{8pt}
	\begin{subfigure}[t]{0.95\textwidth}
		\includegraphics[width=\textwidth]{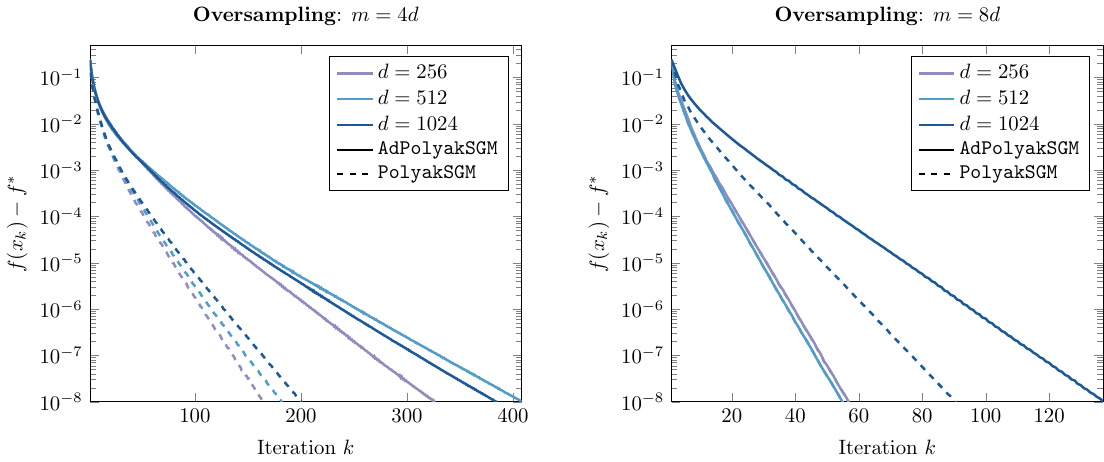}
		\caption{Performance of~\cref{alg:polyaksgm} with $\eta = 1$ against~\cref{alg:adpolyak}
			for measurement vectors $a_i$, $i = 1, \dots, m$ generated from a randomized Walsh-Hadamard (RWHT)
			ensemble.}
		\label{fig:adaptive-vs-nonadaptive-rfwht}
	\end{subfigure}
	\caption{%
		Comparison between~\cref{alg:polyaksgm} and~\cref{alg:adpolyak} for synthetic problems
		with Gaussian and randomized Walsh-Hadamard measurements.
	}
	\label{fig:adaptive-vs-nonadaptive}
\end{figure}

\subsection{Effect of measurement noise}
\label{sec:experiments-noisy}
This section examines the performance of the proposed method in the presence of
measurement noise. We start by discussing necessary algorithmic modifications:
first, note that the presence of noise implies that the optimal value of~\eqref{eq:nonsmooth-loss}
is no longer $f_{\star} = 0$.
While it is possible in theory to apply~\cref{alg:polyaksgm} to
problems with arbitrary optimal value $f_{\star}$, it is rarely the case that $f_{\star}$ is known or
easy to estimate a-priori. To remedy this,
one can replace the Polyak step with a geometrically decaying stepsize,
as suggested by~\citet{Polyak69,Goffin77,DDMP18}, or use the inner-outer loop
structure proposed by~\citet{HK19} for problems
where a lower bound $\widetilde{f} \leq f_{\star}$ is known. Since
$\widetilde{f} = 0$ is always a valid lower bound in our setting,
we opt for the latter approach; see~\cref{sec:appendix:polyaksgm-noopt}
for a detailed description of~\cref{alg:polyaksgm-noopt}, which
we use for these experiments.

Our experiments follow the Poisson noise model.
This model is relevant for \emph{photon-counting CT},
where instead of the energy of the beam incident to the detector at angle
$\theta_i$ we measure the raw number of particles reaching the detector. Herein,
measurements satisfy
\[
	y_{i} \sim \mathrm{Poisson}\left(
	S_{i} \cdot \expfun{-\ip{a_i, \sol}_+}
	\right), \quad i = 1, \dots, m,
\]
where $S_{i} > 0$ is determined by the sensitivity of the detector and
the intensity of the X-ray beam. For large $S_i$, we can use the following Gaussian
approximation:
\begin{equation}
	y_i \approx
	S_i \expfun{-\ip{a_i, \sol}_+} +
	\sqrt{S_i \expfun{-\ip{a_i, \sol}_+}} \cdot \xi_i, \;\;
	\xi_i \sim \cN(0, 1).
	\label{eq:gaussian-approx-to-poisson}
\end{equation}
Our Poisson simulations use the model~\eqref{eq:gaussian-approx-to-poisson}
and fix $S_{i} \equiv S = 10^{5}$ for all detectors.

In what follows, we run a synthetic experiment with $d = 128$ (Gaussian ensemble), $d = 128 \times 128$ (\texttt{RWHT} ensemble),
and $m = 8 \cdot d$, while varying the norm of the unknown signal $\norm{\sol} \in \set{2, 3, 4, 5}$. We then run~\cref{alg:polyaksgm-noopt}
with $T_{\mathsf{outer}} = 10$, $T_{\mathsf{inner}} = 1000$, $x_{0} = \bm{0}$
and $\eta = 1$, for a total budget of $T_{\mathsf{outer}} \times T_{\mathsf{inner}} = 10^{4}$
subgradient evaluations, and compare against gradient descent with a budget of
$10^{4}$ iterations starting from the same initial point $x_0$. For each run of
gradient descent, we optimize over a logarithmically-spaced grid $\set{2^{j} \mid j = -3, -2, \dots, 2, 3}$
to pick the step-size $\eta_k$. The results of the experiment are shown in~\cref{fig:noisy},
illustrating significant performance gains for~\cref{alg:polyaksgm-noopt} compared
to gradient descent across all configurations.

\begin{figure}[h]
	\centering
	\includegraphics[width=.95\linewidth]{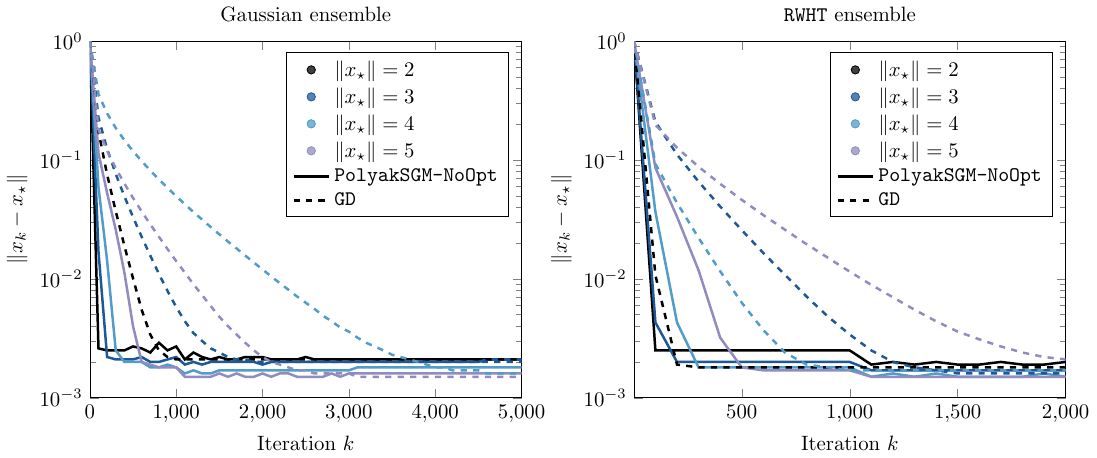}
	\caption{Convergence of Alg.~\ref{alg:polyaksgm-noopt} and
		and gradient descent (\texttt{GD}) for problems with Poisson measurement noise~\eqref{eq:gaussian-approx-to-poisson}.
		For a fixed budget of (sub)gradient evaluations, the two methods achieve similar estimation errors.
		However,~\cref{alg:polyaksgm-noopt} typically
		converges to the minimal distance an order of magnitude faster than gradient descent.}
	\label{fig:noisy}
\end{figure}

\subsection{Experiments on CT imaging}
We now turn to a more realistic model of X-ray CT imaging, wherein the
unknown signal $\sol$ is a vectorized version of an image $I_{\star} \in [0, 1]^{n \times n}$
and $a_{i}$ is the $i^{\text{th}}$ row of the discrete Radon transform applied
to $\sol$ at angle $\theta_i$.
In a typical imaging setup, the number of measurements $m \ll n^2$ and the
resulting problem is ill-posed, necessitating the use of regularization
in the form of constraints or penalties. In this section, we focus on
total-variation (TV) regularization:\footnote{
	We focus on the constrained instead of the penalty version of TV regularization,
	since the Polyak stepsize cannot be used as-is in the presence of regularization
	penalties.
}
\begin{align}
	\min_{x} & \set{\frac{1}{m} \sum_{i = 1}^m \abs{y_i - h_i(x)} \mid \tvnorm{\mat(x)} \leq \lambda},
	\label{eq:tv-regularization}
\end{align}
where $\mat(x): \Rbb^{n^2} \to \Rbb^{n \times n}$ reshapes a vector to a square
matrix of compatible shape and $\tvnorm{x}$ is the total variation norm
proposed by~\cite{ROF92}. Given $x \in \Rbb^{n^2}$ and $X \in \Rbb^{n \times n}$,
\begin{align}
	\mat(x)    & = \begin{bmatrix}
		               x_{1:n}      &
		               x_{(n+1):2n} &
		               \dots        &
		               x_{(n^2 - n + 1):n^2}
	               \end{bmatrix},
	\label{eq:mat-operator}             \\
	\tvnorm{X} & := \sum_{i, j} \sqrt{
		(X_{i + 1, j} - X_{i, j})^2 + (X_{i, j+1} - X_{i, j})^2
	}.
	\label{eq:tvnorm}
\end{align}
To solve~\eqref{eq:tv-regularization}, we use the projected subgradient method;
the steps of~\cref{alg:polyaksgm} become
\begin{equation}
	x_{k+1} := \proj_{\cX}\parens[\bigg]{x_{k} - \eta \frac{f(x_k)}{\norm{v_k}^2} v_k},
	\;\;
	\text{where}
	\;\;
	v_k \in \partial f(x_k), \;\;
	\cX := \set{x \mid \tvnorm{\mat(x)} \leq \lambda}.
	\label{eq:proj-gradient}
\end{equation}
The bottleneck in implementing~\eqref{eq:proj-gradient} is the computation of
the projection operator; we describe a simple iterative approach based on the Douglas-Rachford (DR)
splitting method for computing the projection in~\cref{appendix:sec:tv-norm-ball-projection}.
For our experiments, we use the same iterative method to endow the gradient
descent step in~\eqref{eq:gd-baseline} with a projection onto
$\cX$.

We now turn to our experiments, where we compare
the reconstruction quality achieved by~\cref{alg:polyaksgm} and gradient descent~\eqref{eq:gd-baseline}
for recovering a $128 \times 128$ Shepp-Logan phantom from CT measurements. We
follow the experimental setup from~\citet{FVW+23}; namely, we enlarge the radius
of one of the center ellipsoids of the Shepp-Logan phantom by a factor of $2$
and adjust the corresponding pixel intensities in the range $\set{0.5, 1.0, 2.0}$ to simulate the
presence of different materials, and scale down all pixel intensity values by $4$. We then run
both methods from the same initialization $x_0 = 0$ for $10000$ iterations and
visualize the reconstructions at $k = 1000$, $k = 5000$
and $k = 10000$ iterations.
We also report the peak signal-to-noise ratio (\texttt{PSNR}) achieved by each reconstruction method; given
a ground truth image $I_{\star} \in [0, 1]^{n_1\times n_2}$ and a reconstruction $\widehat{I} \in [0, 1]^{n_1 \times n_2}$,
the \texttt{PSNR} of $\widehat{I}$ is
\[
	\texttt{PSNR}(\widehat{I}; I) := -10 \log_{10}\left(
	\frac{1}{n_1 \cdot n_2}
	\left( \frac{\frobnorm{\widehat{I} - I}}{\max_{i, j} \abs{I_{ij}}} \right)^2
	\right).
\]
The reconstruction results are shown in~\Cref{fig:shepp-logan-low} for the center ellipsoid with low intensity
and~\Cref{fig:shepp-logan-low-high} for the center ellipsoid with high intensity.\footnote{
	To better illustrate the differences in reconstruction, we plot the square root of pixel intensity values.
}
Qualitatively, the image found by gradient descent is much blurrier than that of~\cref{alg:polyaksgm}
and is unable to resolve the different ellipsoids in the earlier stages of reconstruction (\cref{fig:shepp-logan-100,fig:shepp-logan-500,fig:shepp-logan-high-1000,fig:shepp-logan-high-5000}).
Additionally, the gradient descent reconstruction exhibits streak artifacts even for low pixel intensity values in the
central ellipsoid, while~\cref{alg:polyaksgm} only exhibits
limited streak artifacts in the high-intensity setup.
Quantitatively, the reconstruction \texttt{PSNR} achieved by~\cref{alg:polyaksgm} is
consistently much higher than that of gradient descent, with a $20\%$ relative improvement at the final stage.

\begin{figure}
	\begin{subfigure}[t]{\textwidth}
		\centering
		\begin{subfigure}[t]{0.3\textwidth}
			\includegraphics[width=\linewidth]{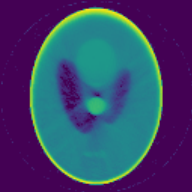}
			\caption{\texttt{GD}; $\texttt{PSNR} = 24.273$}
		\end{subfigure}~
		\begin{subfigure}[t]{0.3\textwidth}
			\includegraphics[width=\linewidth]{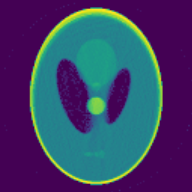}
			\caption{Alg.~\ref{alg:polyaksgm}; $\texttt{PSNR} = 27.470$}
		\end{subfigure}~
		\begin{subfigure}[t]{0.3\textwidth}
			\includegraphics[width=\linewidth]{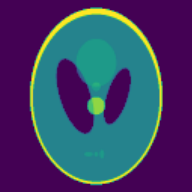}
			\caption{Ground truth}
		\end{subfigure}
		\addtocounter{subfigure}{-3}
		\renewcommand{\thesubfigure}{\roman{subfigure}}
		\caption{Reconstructions at $k = 1000$ iterations}
		\label{fig:shepp-logan-100}
	\end{subfigure}

	\vspace{12pt}
	\addtocounter{subfigure}{-1}

	\begin{subfigure}[t]{\textwidth}
		\centering
		\begin{subfigure}[t]{0.3\textwidth}
			\includegraphics[width=\linewidth]{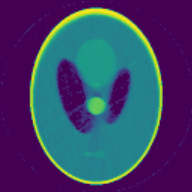}
			\caption{\texttt{GD}; $\texttt{PSNR} = 29.045$}
		\end{subfigure}~
		\begin{subfigure}[t]{0.3\textwidth}
			\includegraphics[width=\linewidth]{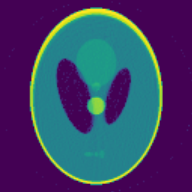}
			\caption{Alg.~\ref{alg:polyaksgm}; $\texttt{PSNR} = 33.655$}
		\end{subfigure}~
		\begin{subfigure}[t]{0.3\textwidth}
			\includegraphics[width=\linewidth]{figures/shepp-logan-img-low/phantom_orig.pdf}
			\caption{Ground truth}
		\end{subfigure}
		\addtocounter{subfigure}{-2}
		\renewcommand{\thesubfigure}{\roman{subfigure}}
		\caption{Reconstructions at $k = 5000$ iterations}
		\label{fig:shepp-logan-500}
	\end{subfigure}

	\vspace{12pt}
	\addtocounter{subfigure}{-2}

	\begin{subfigure}[t]{\textwidth}
		\centering
		\begin{subfigure}[t]{0.3\textwidth}
			\includegraphics[width=\linewidth]{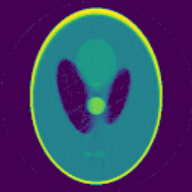}
			\caption{\texttt{GD}; $\texttt{PSNR} = 31.773$}
		\end{subfigure}~
		\begin{subfigure}[t]{0.3\textwidth}
			\includegraphics[width=\linewidth]{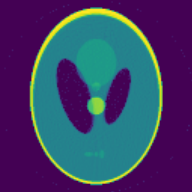}
			\caption{Alg.~\ref{alg:polyaksgm}; $\texttt{PSNR} = 38.115$}
		\end{subfigure}~
		\begin{subfigure}[t]{0.3\textwidth}
			\includegraphics[width=\linewidth]{figures/shepp-logan-img-low/phantom_orig.pdf}
			\caption{Ground truth}
		\end{subfigure}
		\addtocounter{subfigure}{-1}
		\renewcommand{\thesubfigure}{\roman{subfigure}}
		\caption{Final reconstructions ($k = 10000$ iterations)}
		\label{fig:shepp-logan-2000}
	\end{subfigure}
	\vspace{6pt}
	\caption{Comparison of reconstructions produced by gradient descent~\eqref{eq:gd-baseline}, labeled \texttt{GD}, and~\cref{alg:polyaksgm}
		for the modified Shepp-Logan phantom with low pixel intensity ($0.5)$ in the center ellipsoid.
		\Cref{alg:polyaksgm} is run with $\eta = 1$, while \texttt{GD} is
		run with $\eta_0$ as specified in~\eqref{eq:gd-baseline} and $\eta_{k} \equiv \eta_{\mathsf{best}}$ for $k \geq 1$, where $\eta_{\mathsf{best}}$
		is determined by optimizing over a logarithmically-spaced grid of values $\set{2^{j} \mid -3 \leq j  \leq 3}$. Both algorithms
		were initialized at $x_0 = 0$.
	}
	\label{fig:shepp-logan-low}
\end{figure}

\begin{figure}
	\begin{subfigure}[t]{\textwidth}
		\centering
		\begin{subfigure}[t]{0.3\textwidth}
			\includegraphics[width=\linewidth]{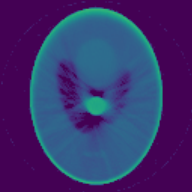}
			\caption{\texttt{GD}; $\texttt{PSNR} = 24.998$}
		\end{subfigure}~
		\begin{subfigure}[t]{0.3\textwidth}
			\includegraphics[width=\linewidth]{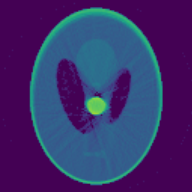}
			\caption{Alg.~\ref{alg:polyaksgm}; $\texttt{PSNR} = 29.055$}
		\end{subfigure}~
		\begin{subfigure}[t]{0.3\textwidth}
			\includegraphics[width=\linewidth]{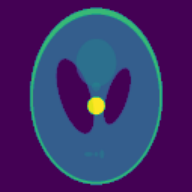}
			\caption{Ground truth}
		\end{subfigure}
		\addtocounter{subfigure}{-3}
		\renewcommand{\thesubfigure}{\roman{subfigure}}
		\caption{Reconstructions at $k = 1000$ iterations}
		\label{fig:shepp-logan-high-1000}
	\end{subfigure}

	\vspace{12pt}
	\addtocounter{subfigure}{-1}

	\begin{subfigure}[t]{\textwidth}
		\centering
		\begin{subfigure}[t]{0.3\textwidth}
			\includegraphics[width=\linewidth]{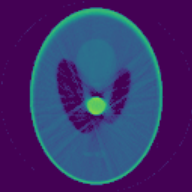}
			\caption{\texttt{GD}; $\texttt{PSNR} = 28.652$}
		\end{subfigure}~
		\begin{subfigure}[t]{0.3\textwidth}
			\includegraphics[width=\linewidth]{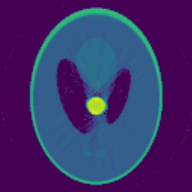}
			\caption{Alg.~\ref{alg:polyaksgm}; $\texttt{PSNR} = 34.906$}
		\end{subfigure}~
		\begin{subfigure}[t]{0.3\textwidth}
			\includegraphics[width=\linewidth]{figures/shepp-logan-img-high/phantom_orig.pdf}
			\caption{Ground truth}
		\end{subfigure}
		\addtocounter{subfigure}{-2}
		\renewcommand{\thesubfigure}{\roman{subfigure}}
		\caption{Reconstructions at $k = 5000$ iterations}
		\label{fig:shepp-logan-high-5000}
	\end{subfigure}

	\vspace{12pt}
	\addtocounter{subfigure}{-2}

	\begin{subfigure}[t]{\textwidth}
		\centering
		\begin{subfigure}[t]{0.3\textwidth}
			\includegraphics[width=\linewidth]{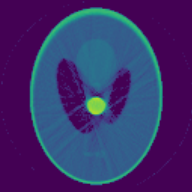}
			\caption{\texttt{GD}; $\texttt{PSNR} = 30.825$}
		\end{subfigure}~
		\begin{subfigure}[t]{0.3\textwidth}
			\includegraphics[width=\linewidth]{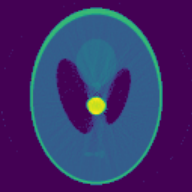}
			\caption{Alg.~\ref{alg:polyaksgm}; $\texttt{PSNR} = 37.894$}
		\end{subfigure}~
		\begin{subfigure}[t]{0.3\textwidth}
			\includegraphics[width=\linewidth]{figures/shepp-logan-img-high/phantom_orig.pdf}
			\caption{Ground truth}
		\end{subfigure}
		\addtocounter{subfigure}{-1}
		\renewcommand{\thesubfigure}{\roman{subfigure}}
		\caption{Final reconstructions ($k = 10000$ iterations)}
		\label{fig:shepp-logan-high-10000}
	\end{subfigure}
	\vspace{6pt}
	\caption{Comparison of reconstructions produced by gradient descent~\eqref{eq:gd-baseline}, labeled \texttt{GD}, and~\cref{alg:polyaksgm}
		for the modified Shepp-Logan phantom with high pixel intensity ($2.0$) in the center ellipsoid.
		\Cref{alg:polyaksgm} is run with $\eta = 1$, while \texttt{GD} is
		run with $\eta_0$ as specified in~\eqref{eq:gd-baseline} and $\eta_{k} \equiv \eta_{\mathsf{best}}$ for $k \geq 1$, where $\eta_{\mathsf{best}}$ is determined
		by optimizing over a logarithmically-spaced grid of values $\set{2^{j} \mid -3 \leq j  \leq 3}$. Both algorithms
		were initialized at $x_0 = 0$.
	}
	\label{fig:shepp-logan-low-high}
\end{figure}

\clearpage

\section{Discussion}
\label{sec:discussion}
There are several open questions and potential extensions to our work that we now discuss.

\begin{enumerate}
	\item Our analysis assumes Gaussian observations, but design vectors corresponding
	      to real CT measurements are typically sparse, nonnegative and highly structured;
	      obtaining efficiency estimates in this setting is likely challenging. Nevertheless,
	      it may be interesting to consider structured designs with limited randomness,
	      such as random illumination angles.
	\item Our theoretical results do not consider noisy measurements, such as those
	      in the experiments of~\cref{sec:experiments-noisy}.
	      The Poisson noise model, wherein the
	      observed energy satisfies
	      \[
		      y_i \sim \mathrm{Poisson}(S \exp(-\ip{a_i, \sol}_+)),
	      \]
	      is of most relevance to photon-counting CT problems.
	      We leave a theoretical investigation of the effects of different types of noise
	      to future work.
	\item Our current model does not take into account the
	      polychromatic nature of the X-ray beam, which is typically distributed across
	      a spectrum of energies~\cite{EF03,BSSP16}. This is an important modeling step, since attenuation
	      coefficients of materials used in CT imaging typically vary across the energy
	      spectrum. To address this, one could consider measurements of the form
	      \[
		      y_{i} = \sum_{j = 1}^{n_{E}} S_{j} \expfun{-\zeta_{j} \ip{a_i, \sol}_+}, \;\; i = 1, \dots, m,
	      \]
	      where $S_{j}$ is the intensity of the beam at energy level $j$, $n_{E}$ is the
	      number of distinct energy levels considered -- resulting from a discretization of the
	      energy spectrum -- and $\zeta_{j}$ models a different known attenuation coefficient per energy level.
	\item Empirically, we observe that the Polyak stepsize without any modification
	      (i.e.,~\cref{alg:polyaksgm} with $\eta = 1$)
	      performs best across all instances. A potential explanation is that the optimization
	      problem~\eqref{eq:nonsmooth-loss} possesses ``benign landscape''
	      around $x_{0} = \bm{0}$, so that a few Polyak steps guide iterates towards
	      a local basin of attraction, wherein the conditioning of the loss function
	      is much better than our estimate. While it is
	      unrealistic to expect the condition number to be independent of $\norm{\sol}$,
	      obtaining the ``optimal'' dependence is an interesting direction for future work.
	\item Finally, another potential direction for future research is expanding the
	      algorithmic toolkit applied to the problem~\eqref{eq:nonsmooth-loss}.
	      For example, it may be desirable to solve~\eqref{eq:nonsmooth-loss} with
	      with stochastic methods and/or using better local models for the objective
	      function~\cite{AD19,DDC23}.
\end{enumerate}

\subsection*{Acknowledgements}
We thank Damek Davis and Emil Sidky for helpful discussions.
VC and RW gratefully acknowledge the support of AFOSR FA9550-18-1-0166,
NSF DMS-2023109, DOE DE-SC0022232 and the Margot and Tom Pritzker Foundation.

\bibliography{references}

\clearpage

\appendix

\section{Omitted proofs}

\label{appendix:sec:missing-proofs}

\subsection{Proofs from~\cref{sec:regularity-properties}}

\subsubsection{Proof of~\cref{lemma:loss-concentration-nosimpl}}
\begin{proof}
	Note that when $x = \bm{0}$, $y_i - h_{i}(0) = 1 - \expfun{-\ip{a_i, \sol}_+}$.
	As a result,
	\begin{align}
		f(0) & = \frac{1}{m} \sum_{i = 1}^m 1 - \expfun{-\ip{a_i, \sol}_+} \notag \\
		     & = 1 - \frac{1}{m} \sum_{i = 1}^m \expfun{-\norm{\sol}
			\ip{a_i, \frac{\sol}{\norm{\sol}}}_+
		}                                                                  \notag \\
		     & \overset{\mathclap{(d)}}{=}
		1 - \frac{1}{m} \sum_{i=1}^m \expfun{-(\beta_i)_+ \norm{\sol}}, \qquad
		\beta_i \sim \cN(0, 1).
		\label{eq:f0-sum-distribution}
	\end{align}
	We now argue the sum in~\eqref{eq:f0-sum-distribution}  concentrates. To that end, we first calculate its expectation:
	\begin{align*}
		\mathbb{E}[f(0)] & = 1
		- \mathbb{E}[e^{-\beta_+ \norm{\sol}}] \\
		                 & =
		1 - \frac{1}{2} \left(
		1 + \expfun{\frac{\norm{\sol}^2}{2}} \erfc\left(\frac{\norm{\sol}}{\sqrt{2}}\right)
		\right)                                \\
		                 & =
		\frac{1}{2} \left(1 - \expfun{\frac{\norm{\sol}^2}{2}} \erfc\left(\frac{\norm{\sol}}{\sqrt{2}}\right) \right).
	\end{align*}
	To show the sum concentrates, we use the Gaussian Lipschitz inequality. We have
	\begin{align*}
		 & \frac{1}{m} \sum_{i = 1}^m (1 - \exp(-\ip{a_i, \sol}_+)) -
		\frac{1}{m} \sum_{i = 1}^m (1 - \exp(-\ip{\widetilde{a}_i, \sol}_+)) \\
		 & =
		\frac{1}{m} \sum_{i = 1}^m
		\exp(-\ip{\widetilde{a}_i, \sol}_+) - \exp(-\ip{a_i, \sol}_+)        \\
		 & \leq
		\frac{1}{m} \sum_{i = 1}^m \abs{\ip{a_i - \widetilde{a}_i, \sol}}    \\
		 & \leq
		\frac{\norm{\sol}}{\sqrt{m}} \frobnorm{A - \widetilde{A}},
		\qquad
		\text{where} \quad
		A = \begin{bmatrix}
			    a_1 & \dots & a_{m}
		    \end{bmatrix}^{\T}, \;
		\widetilde{A} := \begin{bmatrix}
			                 \widetilde{a}_1 & \dots & \widetilde{a}_m
		                 \end{bmatrix}^{\T},
	\end{align*}
	using the fact that the function $x \mapsto \exp(-x_+)$ is $1$-Lipschitz.
	We deduce that $f(0)$ is Lipschitz with modulus $\frac{\norm{\sol}}{\sqrt{m}}$
	with respect to the random vectors $\set{a_i}_{i=1,\dots,m}$. Invoking the Gaussian Lipschitz
	concentration inequality~\cite[Theorem 5.2.2]{Vershynin18} yields
	\begin{equation}
		\prob{\abs{f(0) - \expec{f(0)}} \geq t} \leq
		2 \expfun{-\frac{mt^2}{2\norm{\sol}^2}}
		\label{eq:f0-gaussian-lipschitz}
	\end{equation}
	Substituting $t = \norm{\sol} \sqrt{\frac{d}{m}}$ into~\eqref{eq:f0-gaussian-lipschitz}
	finishes the proof.
\end{proof}

\subsubsection{Proof of~\cref{corollary:loss-concentration-simpl}}
\begin{proof}
	From~\cref{lemma:erfcx-increasing}, we know that $\exp(\frac{u^2}{2}) \erfc(\frac{u}{\sqrt{2}})$
	is monotone decreasing on $u \in (0, \infty)$. As a result, we can lower bound the expectation by
	\begin{align*}
		\expec{f(0)} & =
		\frac{1}{2} \left(1 - \expfun{\frac{\norm{\sol}^2}{2}} \erfc\left(\frac{\norm{\sol}}{\sqrt{2}}\right)\right) \\
		             & \geq
		\frac{1}{2} \left(1 - \sqrt{\mathrm{e}} \cdot \erfc\left(\frac{1}{ \sqrt{2} }\right)\right)                  \\
		             & \geq
		\frac{1}{2} (1 - 0.53)                                                                                       \\
		             & = 0.235,
	\end{align*}
	after numerically evaluating $\sqrt{\mathrm{e}} \cdot \erfc(1 / \sqrt{2}) \leq 0.53$.
	The rest follows from~\cref{lemma:loss-concentration-nosimpl}.
\end{proof}

\subsubsection{Proof of~\cref{lemma:subgradient-correlation-at-0}}
\begin{proof}
	We first argue that $v_0$ is indeed a subgradient at $0$. To that end, note that
	\begin{align*}
		\partial f(x) & = \frac{1}{m} \sum_{i=1}^m \sign(y_i - h_i(x)) \cdot (-\partial h_{i}(x))                              \\
		              & = \frac{1}{m} \sum_{i = 1}^m \sign(y_i - h_i(x)) (-e^{-\ip{a_i, x}_+}) a_i \1\set{\ip{a_i, x} \geq 0}.
	\end{align*}
	Note $y_i - h_i(0) = 1 - \expfun{-\ip{a_i, \sol}_+} \geq 0$. Using
	$\sign(0) = 0$ and $\1\set{x \geq 0} = 1$ for $x \geq 0$, we have
	the following expression for $v_0$:
	\begin{align*}
		v_0 = \frac{1}{m} \sum_{i = 1}^m -a_i \1\set{\ip{a_i, \sol} > 0}
		\Rightarrow -\ip{v_0, \sol} & =
		\frac{1}{m} \sum_{i = 1}^m \ip{a_i, \sol}_+
		\overset{\mathclap{(d)}}{=}
		\frac{\norm{\sol}}{m} \sum_{i = 1}^m (\beta_i)_+,
	\end{align*}
	where $\beta_i \iid \cN(0, 1)$. This sum has expected value
	\begin{equation}
		\expec{-\ip{v_0, \sol}} =
		\norm{\sol} \expec[X \sim \cN(0, 1)]{X_+} =
		\norm{\sol} \sqrt{\frac{1}{2 \pi}}.
		\label{eq:v0-sol-expec}
	\end{equation}
	Additionally, it is $(\norm{\sol} / \sqrt{m})$ Lipschitz as a function of $A = (a_1, \dots, a_m)$, since
	\begin{align*}
		\abs{\frac{1}{m} \sum_{i = 1}^m \ip{a_i, \sol}_+ - \ip{\widetilde{a}_i, \sol}_+}
		 & \leq \frac{1}{m} \sum_{i = 1}^m \abs{\ip{a_i - \widetilde{a}_i, \sol}} \\
		 & \leq
		\frac{\norm{\sol}}{m} \sum_{i=1}^m \norm{a_i - \widetilde{a}_i}           \\
		 & \leq
		\frac{\norm{\sol}}{\sqrt{m}} \frobnorm{A - \widetilde{A}},
	\end{align*}
	where the first inequality follows from the fact that $x \mapsto [x]_+$ is $1$-Lipschitz,
	the second inequality follows from Cauchy-Schwarz and the last inequality follows
	from norm equivalence. Applying the Gaussian concentration inequality~\cite[Theorem 5.2.2]{Vershynin18},
	we obtain
	\begin{equation*}
		\prob{\abs{-\ip{v_0, \sol} - \norm{\sol} \frac{1}{\sqrt{2\pi}}} \geq t}
		\leq 2\expfun{-\frac{mt^2}{2 \norm{\sol}^2}}
	\end{equation*}
	Setting $t := \norm{\sol} \sqrt{\frac{2d}{m}}$ yields~\eqref{eq:v0-correlation}.
	For~\eqref{eq:v0-x-norm}, we write
	\begin{align*}
		\norm{v_0} & = \sup_{u \in \Sbb^{d-1}}
		\frac{1}{m} \sum_{i = 1}^m \ip{a_i, u} \1\set{\ip{a_i, \sol} \geq 0} \\
		           & \leq
		\sup_{u \in \Sbb^{d-1}}
		\frac{1}{m} \sum_{i = 1}^m \abs{\ip{a_i, u}}                         \\
		           & \leq
		\frac{1}{\sqrt{m}} \opnorm{A},
	\end{align*}
	which is bounded from above by $1 + 2 \sqrt{\frac{d}{m}}$ with probability
	at least $1 - \exp(-d)$ by~\cite[Corollary 7.3.3]{Vershynin18}. This
	completes the proof of~\eqref{eq:v0-x-norm}. Finally, for~\eqref{eq:initial-distance}, we write
	\begin{align*}
		\norm{x_1 - \sol}^2 & =
		\norm{\sol}^2 + \frac{\eta^2 f^2(0)}{\norm{v_0}^2}
		- 2 \eta \frac{f(0)}{\norm{v_0}^2} \ip{v_0, 0 - \sol}                 \\
		                    & =
		\norm{\sol}^2 - \frac{\eta f(0)}{\norm{v_0}^2} \left(
		-2 \ip{v_0, \sol} - \eta f(0)
		\right)                                                               \\
		                    & \leq
		\norm{\sol}^2 - \frac{\eta f(0)}{\norm{v_0}^2} \left(
		\frac{2}{\sqrt{\pi}} \norm{\sol} - \eta
		\right)                                                               \\
		                    & \leq
		\norm{\sol}^2 - \frac{\eta f(0)}{\norm{v_0}^2 \sqrt{\pi}} \norm{\sol} \\
		                    & \leq
		\norm{\sol}^2 \left(
		1 - \frac{\eta}{10 \sqrt{\pi} \norm{\sol}}
		\right),
	\end{align*}
	where the first inequality follows from~\eqref{eq:v0-correlation} and $f(0) \leq 1$,
	the second inequality follows from the assumptions $\eta \leq \frac{1}{2}$ and
	$\norm{\sol} \geq 1$, and the last inequality follows from~\cref{corollary:loss-concentration-simpl}
	and~\eqref{eq:v0-x-norm}.
	This completes the proof of~\eqref{eq:initial-distance}.
\end{proof}

\subsubsection{Proof of~\cref{prop:lipschitz}}
\begin{proof}
	Recall that the absolute value function and $x \mapsto \exp(-x_+)$ are $1$-Lipschitz. It follows that,
	for any pair $x, \bar{x} \in \Rbb^d$, we have
	\begin{align*}
		f(x) - f(\bar{x}) & =
		\frac{1}{m} \sum_{i = 1}^m \abs{y_i - h_i(x)} - \abs{y_i - h_i(\bar{x})}                \\
		                  & \leq
		\frac{1}{m} \sum_{i = 1}^m \abs{h_i(x) - h_i(\bar{x})}                                  \\
		                  & =
		\frac{1}{m} \sum_{i = 1}^m \abs{e^{-\ip{a_i, x}_+} - e^{-\ip{a_i, \bar{x}}_+}}          \\
		                  & \leq
		\frac{1}{m} \sum_{i = 1}^m \abs{\ip{a_i, x - \bar{x}}}                                  \\
		                  & \leq
		\norm{x - \bar{x}} \cdot \sup_{u \in \Sbb^{d-1}} \frac{1}{m} \sum_{i = 1}^m \norm{Au}_1 \\
		                  & \leq
		\norm{x - \bar{x}} \cdot \sup_{u \in \mathbb{S}^{d-1}} \frac{1}{\sqrt{m}} \norm{Au}     \\
		                  & =
		\norm{x - \bar{x}} \cdot \frac{1}{\sqrt{m}} \opnorm{A}.
	\end{align*}
	A completely symmetric argument yields $f(\bar{x}) - f(x) \leq \norm{x - \bar{x}} \cdot \frac{\opnorm{A}}{\sqrt{m}}$.
	From~\cite[Corollary 7.3.3]{Vershynin18}, it follows that $\opnorm{A} \leq \sqrt{m} + 2 \sqrt{d}$ with
	probability at least $1 - \expfun{-d}$. Since the latter event does not depend on the
	choice of $x, \bar{x}$, we conclude that
	\begin{align*}
		\abs{f(x) - f(\bar{x})} \leq \norm{x - \bar{x}} \cdot \left(1 + 2 \sqrt{\frac{d}{m}}\right),
		\quad \text{for all $x, \bar{x} \in \mathbb{R}^d$},
	\end{align*}
	with probability at least $1 - \expfun{-d}$.
\end{proof}

\subsubsection{Proof of~\cref{lemma:sharpness-from-aiming}}
\begin{proof}
	The main stepping stone to~\eqref{eq:sharpness-from-aiming} is the
	following ``decrease principle'' (see~\cite[Theorem 3.2.8]{CLSW08} for a version
	using the so-called \emph{proximal} subdifferential):
	\begin{claim}[Decrease principle]
		\label{claim:decrease-principle}
		Let $f$ be locally Lipschitz and fix $\rho, \mu > 0$. Suppose that
		\begin{equation}
			z \in \cB(x; \rho), \;\; \zeta \in \partial f(z) \implies \norm{\zeta} \geq \mu,
			\label{eq:decrease-principle-condition}
		\end{equation}
		Then the following inequality holds:
		\begin{equation}
			\inf_{z \in \cB(x; \rho)} f(z) \leq f(x) - \rho \mu.
			\label{eq:decrease-principle}
		\end{equation}
	\end{claim}
	Before proving~\cref{claim:decrease-principle}, we show how
	it implies the conclusion in~\eqref{eq:sharpness-from-aiming}. Indeed,
	suppose the conclusion were false; then, for some $x \in \cB(\sol; \norm{\sol})$, there is
	$\rho > 0$ such that
	\[
		\min\left(\norm{x - \sol}, \norm{x}\right) > \rho > \frac{f(x) - f_{\star}}{\mu}.
	\]
	From~\cref{lemma:minimum-irrelevant} and $\norm{x} \leq 2 \norm{\sol}$ it follows that
	\begin{equation}
		\min\left(\norm{x - \sol}, \norm{x}, \dist(x, \cB^c(0; 3 \norm{\sol})\right) = \min(\norm{x - \sol}, \norm{x}) > \rho > \frac{f(x) - f_{\star}}{\mu}.
		\label{eq:solvability-contradiction}
	\end{equation}
	At the same time, since $\min(\cdot)$ is associative, we have that
	\[
		\min(\norm{x}, \dist(x, \cB^c(0; 3 \norm{\sol}))) = \dist\left(x, \left(\cB(0; 3 \norm{\sol}) \setD \set{0}\right)^c\right).
	\]
	Consequently, it follows that
	\[
		\cB(x; \rho) \subset \cB(0; 3 \norm{\sol}) \setD \set{0}, \;\;
		\norm{x - \sol} > \rho.
	\]
	The second conclusion in the display above implies that $x \neq \sol$. Therefore,
	\[
		x \in \cB(0; 3 \norm{\sol}) \setD \set{0, \sol}
		\overset{\eqref{eq:aiming-abstract}}{\implies}
		\min_{v \in \partial f(x)} \norm{v} \geq \mu.
	\]
	As a result, invoking~\cref{claim:decrease-principle}, we obtain
	\begin{equation}
		0 \leq \inf_{\bar{x} \in \cB(x; \rho)} f(\bar{x}) - f_{\star} \leq (f(x) - f_{\star}) - \rho \mu
		\overset{\eqref{eq:solvability-contradiction}}{<} 0,
	\end{equation}
	which is a contradiction with our assumption that~\eqref{eq:sharpness-from-aiming} fails;
	therefore, \eqref{eq:sharpness-from-aiming} must hold.

	\begin{proof}[Proof of~\cref{claim:decrease-principle}]
		It remains to prove~\cref{claim:decrease-principle}. Since
		this is essentially the same as~\cite[Theorem 3.2.8]{CLSW08}, with the proximal subdifferential
		replaced by the Clarke subdifferential, it suffices to repeat its proof with a single
		modification: instead of applying the version of the mean-value inequality from~\cite[Theorem 3.2.6]{CLSW08},
		we invoke~\cite[Theorem 4.1]{CL94}, which is valid for the Clarke subdifferential.
	\end{proof}
	This completes the proof of the Lemma.
\end{proof}

\subsubsection{Proof of~\cref{lemma:aiming-inner-product}}
\begin{proof}
	Recall that $\sign(x) = 1$ for $x > 0$, $-1$ for $x < 0$, and $\sign(0) = [-1, 1]$.
	In turn,
	\begin{align}
		 & \sign(y_i - h_i(x)) \notag                            \\
		 & =
		\sign(e^{-\ip{a_i, x}_+} - e^{-\ip{a_i, \sol}_+}) \notag \\
		 & =
		\sign(e^{-\ip{a_i, x}_+} - e^{-\ip{a_i, \sol}_+}) \cdot
		\left(
		\1\set{\ip{a_i, x - \sol} \ge 0} +
		\1\set{\ip{a_i, x - \sol} < 0}
		\right)
		\label{eq:sign-indicator-decomp}
	\end{align}
	Note that it suffices to consider $\1\set{\ip{a_i, x - \sol} > 0}$ in the first term of~\eqref{eq:sign-indicator-decomp}, since
	\begin{align*}
		\frac{1}{m} \sum_{i = 1}^m \sign(y_i - h_i(x))e^{-\ip{a_i, x}} \ip{a_i, x - \sol} \1\set{\ip{a_i, x} \geq 0} \1\set{\ip{a_i, x - \sol} = 0} = 0.
	\end{align*}
	We now proceed on a case-by-case basis.

	\paragraph{Case 1: $\ip{a_i, x - \sol} < 0$.}
	Since all nonzero terms have $\ip{a_i, x} \geq 0$, this means
	\[
		0 \leq \ip{a_i, x} < \ip{a_i, \sol} \implies
		\ip{a_i, \sol}_+ > \ip{a_i, x}_+ \implies
		\sign(e^{-\ip{a_i, x}_+} - e^{-\ip{a_i, \sol}_+}) = 1,
	\]
	via monotonicity of the exponential. This yields the partial sum
	\begin{align}
		 & \frac{1}{m} \sum_{i = 1}^m \sign(y_i - h_i(x))(-e^{-\ip{a_i, x}}) \ip{a_i, x - \sol} \1\set{\ip{a_i, x} \geq 0} \1\set{\ip{a_i, x - \sol} < 0} \notag \\
		 & =
		\frac{1}{m} \sum_{i = 1}^m e^{-\ip{a_i, x}} \ip{a_i, \sol - x} \1\set{\ip{a_i, x} \geq 0} \1\set{\ip{a_i, \sol - x} > 0} \notag                          \\
		 & = \frac{1}{m} \sum_{i = 1}^m e^{-\ip{a_i, x}} \ip{a_i, \sol - x}_+ \1\set{\ip{a_i, x} \geq 0}.
		\label{eq:partial-sum-I}
	\end{align}

	\paragraph{Case 2(i): $\ip{a_i, x - \sol} > 0$ and $\ip{a_i, x} > 0$.}
	Note that we have the following possibilities:
	\begin{itemize}
		\item If $\ip{a_i, \sol} < 0$, then $\ip{a_i, \sol}_+ = 0$. Therefore, $\ip{a_i, x} > 0$ clearly implies
		      \[
			      \ip{a_i, x}_+ > \ip{a_i, \sol}_+, \quad \text{ and thus } \quad
			      \sign(e^{-\ip{a_i, x}_+} - e^{-\ip{a_i, \sol}_+}) = -1.
		      \]
		\item If $\ip{a_i, \sol} \geq 0$, then $\ip{a_i, \sol}_+ = \ip{a_i, \sol}$; therefore,
		      \[
			      \ip{a_i, x}_+ = \ip{a_i, x} > \ip{a_i, \sol} = \ip{a_i, \sol}_+
			      \implies \sign(e^{-\ip{a_i, x}_+} - e^{-\ip{a_i, \sol}_+}) = -1.
		      \]
	\end{itemize}
	In either instance, we obtain the partial sum
	\begin{align}
		 & \frac{1}{m} \sum_{i = 1}^m \sign(y_i - h_i(x))(-e^{-\ip{a_i, x}}) \ip{a_i, x - \sol} \1\set{\ip{a_i, x} > 0} \1\set{\ip{a_i, x - \sol} > 0} \notag \\
		 & =
		\frac{1}{m} \sum_{i = 1}^m e^{-\ip{a_i, x}} \ip{a_i, x - \sol} \1\set{\ip{a_i, x} > 0} \1\set{\ip{a_i, x - \sol} > 0} \notag                          \\
		 & = \frac{1}{m} \sum_{i = 1}^m e^{-\ip{a_i, x}} \ip{a_i, x - \sol}_+ \1\set{\ip{a_i, x} > 0}.
		\label{eq:partial-sum-II-i}
	\end{align}

	\paragraph{Case 2(ii): $\ip{a_i, x - \sol} > 0$ and $\ip{a_i, x} = 0$.}
	In this case, we have
	\[
		0 = \ip{a_i, x} > \ip{a_i, \sol} \implies
		\ip{a_i, x}_+ = \ip{a_i, \sol}_+ = 0.
	\]
	As a result, $\sign(y_i - h_i(x))$ can be any value between $[-1, 1]$. We lower bound
	\begin{align}
		 & \frac{1}{m} \sum_{i = 1}^m \sign(y_i - h_i(x))(-e^{-\ip{a_i, x}}) \ip{a_i, x - \sol} \1\set{\ip{a_i, x} = 0} \1\set{\ip{a_i, x - \sol} > 0} \notag \\
		 & =
		\frac{1}{m} \sum_{i = 1}^m \sign(0) \ip{a_i, \sol} \1\set{\ip{a_i, x} = 0} \1\set{\ip{a_i, x - \sol} > 0} \notag                                      \\
		 & =
		\frac{1}{m} \sum_{i = 1}^m \sign(0) \cdot \ip{a_i, \sol}_{-} \1\set{\ip{a_i, x} = 0},
		\label{eq:partial-sum-II-ii}
	\end{align}
	using $\ip{a_i, \sol} \1\set{\ip{a_i, \sol} < 0} = \ip{a_i, \sol}_{-}$ in~\eqref{eq:partial-sum-II-ii}.

	\vspace{11pt}

	\noindent These are all the possible cases to consider. Combining~\cref{eq:partial-sum-I,eq:partial-sum-II-i,eq:partial-sum-II-ii} yields
	\begin{align*}
		\ip{\partial f(x), x - \sol} & = \frac{1}{m} \sum_{i = 1}^m \sign(y_i - h_i(x)) \cdot (-e^{-\ip{a_i, x}}) \cdot \ip{a_i, x - \sol} \1\set{\ip{a_i, x} \geq 0}                                              \\
		                             & = \begin{aligned}[t]
			                                  & \frac{1}{m} \sum_{i = 1}^m e^{-\ip{a_i, x}} \ip{a_i, x - \sol} \1\set{\ip{a_i, x} > 0} \1\set{\ip{a_i, x - \sol} > 0}      \\
			                                  & - \frac{1}{m} \sum_{i = 1}^m e^{-\ip{a_i, x}} \ip{a_i, x - \sol} \1\set{\ip{a_i, x} \geq 0} \1\set{\ip{a_i, x - \sol} < 0} \\
			                                  & + \frac{1}{m} \sum_{i=1}^m \sign(0) \ip{a_i, \sol}_{-} \1\set{\ip{a_i, x} = 0}
		                                 \end{aligned}                \\
		                             & = \begin{aligned}[t]
			                                  & \frac{1}{m} \sum_{i = 1}^m e^{-\ip{a_i, x}} \ip{a_i, x - \sol} \1\set{\ip{a_i, x} > 0} \1\set{\ip{a_i, x - \sol} \geq 0}   \\
			                                  & + \frac{1}{m} \sum_{i = 1}^m e^{-\ip{a_i, x}} \ip{a_i, \sol - x} \1\set{\ip{a_i, x} \geq 0} \1\set{\ip{a_i, \sol - x} > 0} \\
			                                  & + \frac{1}{m} \sum_{i=1}^m \sign(0) \ip{a_i, \sol}_{-} \1\set{\ip{a_i, x} = 0}
		                                 \end{aligned}                \\
		                             & = \begin{aligned}[t]
			                                  & \frac{1}{m} \sum_{i = 1}^m e^{-\ip{a_i, x}} \ip{a_i, x - \sol}_+ \1\set{\ip{a_i, x} > 0}      \\
			                                  & + \frac{1}{m} \sum_{i = 1}^m e^{-\ip{a_i, x}} \ip{a_i, \sol - x}_+ \1\set{\ip{a_i, x} \geq 0} \\
			                                  & + \frac{1}{m} \sum_{i=1}^m \sign(0) \ip{a_i, \sol}_{-} \1\set{\ip{a_i, x} = 0}
		                                 \end{aligned}                                                  \\
		                             & = \begin{aligned}[t]
			                                  & \frac{1}{m} \sum_{i = 1}^m e^{-\ip{a_i, x}} \1\set{\ip{a_i, x} > 0} \left(\ip{a_i, x - \sol}_+ + \ip{a_i, \sol - x}_+ \right) \\
			                                  & + \frac{1}{m} \sum_{i = 1}^m \1\set{\ip{a_i, x} = 0} \left(\ip{a_i, \sol}_+ + \sign(0) \ip{a_i, \sol}_{-}\right)
		                                 \end{aligned} \\
		                             & =
		\begin{aligned}[t]
			 & \frac{1}{m} \sum_{i = 1}^m e^{-\ip{a_i, x}} \1\set{\ip{a_i, x} > 0} \cdot \abs{\ip{a_i, x - \sol}}   \\
			 & + \frac{1}{m} \sum_{i=1}^m \1\set{\ip{a_i, x} = 0} (\ip{a_i, \sol}_+ + \sign(0) \ip{a_i, \sol}_{-}),
		\end{aligned}
	\end{align*}
	where the last equality follows from the identity $[x]_+ + [-x]_+ = \abs{x}$.
\end{proof}

\subsubsection{Proof of~\cref{lemma:aiming-lb}}
\begin{proof}
	We first argue that we can effectively ignore the second term in the expression furnished by~\cref{lemma:aiming-inner-product},
	since it corresponds to a zero-measure event (for any $x \neq 0$). To show this formally,
	we first note that since $\sign(0) = [-1, 1]$ and $\1\set{\ip{a_i, x} = 0} = [0, 1]$ the contribution of the second term is always at least
	\begin{align*}
		 & \frac{1}{m} \sum_{i = 1}^m \1\set{\ip{a_i, x} = 0} \left(\ip{a_i, \sol}_+ + \sign(0) \ip{a_i, \sol}_{-}\right) \geq \frac{1}{m} \sum_{i = 1}^m \strictind\set{\ip{a_i, x} = 0} \ip{a_i, \sol}_{-},
	\end{align*}
	where $\strictind\set{\cE} = 1$ when $\cE$ happens and $0$ otherwise. We now write
	\begin{align*}
		\strictind\set{\ip{a_i, x} = 0} \ip{a_i, \sol}_{-} & = \strictind\set{\ip{a_i, u} = 0}\left(\cancelto{0}{\ip{a_i, u}} \ip{u, \sol} +
		\ip{P_{u^{\perp}} a_i, P_{u^{\perp}} \sol}\right)_{-}                                                                                \\
		                                                   & \overset{(d)}{=}
		\strictind\set{\ip{a_i, u} = 0} \cdot \ip{\widetilde{a}_i, P_{u^{\perp}} \sol}_{-}, \qquad
		\widetilde{a}_i \sim \cN(0, I_d) \indep a_i,
	\end{align*}
	where we write $u = \sfrac{x}{\norm{x}}$, $P_{u^{\perp}} = (I - uu^{\T})$, and use
	the fact that the variables $\ip{a_i, u}$ and $P_{u^{\perp}} a_i$ are uncorrelated (thus independent) Gaussians
	to replace $P_{u^{\perp}} a_i$ with an independent copy $P_{u^{\perp}} \widetilde{a}_i$. From independence,
	it follows that
	\begin{align}
		\expec{\ip{a_i, \sol}_{-} \1\set{\ip{a_i, x} = 0}} & \geq
		\expec{\strictind\set{\ip{a_i, x} = 0} \cdot \ip{a_i, \sol}_{-}}                                                                               \notag   \\
		                                                   & =
		\expec{\strictind\set{\ip{a_i, u} = 0}} \expec{\ip{\widetilde{a}_i, P_{u^{\perp}}\sol}_{-}}                                                    \notag   \\
		                                                   & = \cancel{\prob{\ip{a_i, u} = 0}} \cdot \expec{\ip{\widetilde{a}_i, P_{u^{\perp}}\sol}_{-}} \notag \\
		                                                   & = 0,
		\label{eq:aiming-expec-ignore}
	\end{align}
	where $\prob{\ip{a_i, u} = 0}$ since $a_i$ is standard Gaussian and $u \neq 0$.
	In light of~\eqref{eq:aiming-expec-ignore}, we may ignore the second term in~\cref{lemma:aiming-inner-product}
	in lower-bounding the expectation, since that term has a nonnegative contribution.
	Continuing from the expression furnished by~\cref{lemma:aiming-inner-product}, we obtain
	\begin{align*}
		\ip{\partial f(x), x - \sol} & =
		\frac{1}{m} \sum_{i = 1}^m \exp(-\ip{a_i, x}) \abs{\ip{a_i, \sol - x}} \1\set{\ip{a_i, x} \geq 0} \\
		                             & =
		\frac{\norm{x - \sol}}{m} \sum_{i = 1}^m \exp(-\ip{a_i, u} \norm{x})
		\abs{\ip{a_i, v}} \1\set{\ip{a_i, u} \geq 0}                                                      \\
		                             & \geq
		\frac{\norm{x - \sol}}{m} \sum_{i = 1}^m \exp(-3\ip{a_i, u} \norm{\sol})
		\abs{\ip{a_i, v}} \1\set{\ip{a_i, u} \geq 0}                                                      \\
		                             & \overset{\mathclap{(d)}}{=}
		\frac{\norm{x - \sol}}{m}
		\sum_{i = 1}^m \expfun{-3\beta_i \norm{\sol}}
		\abs{\beta_i \ip{u, v} + \beta_i^{\perp} \norm{P_{u^{\perp}} v}}
		\1\set{\beta_i \geq 0},
	\end{align*}
	writing $u := \frac{x}{\norm{x}}$ and $v := \frac{\sol - x}{\norm{\sol - x}}$. Here,
	$\beta_i, \beta_i^{\perp} \iid \cN(0, 1)$ arise from the decomposition
	\begin{align*}
		\ip{a_i, v} & = \ip{uu^{\T} a_i, v} + \ip{(I - uu^{\T}) a_i, v}                                                                                              \\
		            & = \ip{a_i, u} \ip{u, v} + \ip{(I - uu^{\T}) a_i, v}                                                                                            \\
		            & \overset{\mathclap{(d)}}{=} \ip{a_i, u} \ip{u, v} + \ip{(I - uu^{\T}) \widetilde{a}_i, v} \qquad (a_i \indep \widetilde{a}_i \sim \cN(0, I_d)) \\
		            & = \ip{a_i, u} \ip{u, v} + \ip{\widetilde{a}_i, (I - uu^{\T}) v}                                                                                \\
		            & \overset{\mathclap{(d)}}{=} \beta_i \ip{u, v} + \beta_i^{\perp} \bignorm{(I - uu^{\T}) v},
	\end{align*}
	writing $\beta_i = \ip{a_i, u} \sim \cN(0, 1)$ and using the identity $\ip{\widetilde{a}_i, z} \sim \cN(0, \norm{z}^2)$;
	we also recognize $(I - uu^{\T}) = P_{u^{\perp}}$. We now consider two cases for the correlation $\ip{u, v}$:

	\paragraph{Case 1: $\ip{u, v} \geq 0$.}
	In this case, we may lower bound the sum by the following expression:
	\begin{align}
		 & \ip{\partial f(x), x - \sol}                                                                         \notag         \\ & \geq
		\frac{\norm{x - \sol}}{m} \sum_{ i = 1 }^m
		\exp(-3 \beta_i \norm{\sol})
		\abs{\beta_i \ip{u, v} + \beta_i^{\perp} \norm{P_{u^{\perp}} v}} \1\set{\beta_i \geq 0, \beta_i^{\perp} \geq 0} \notag \\
		 & =
		\frac{\norm{x - \sol}}{m}
		\sum_{i = 1}^m \exp(-3 \beta_i \norm{\sol})
		\left(\beta_i \abs{\ip{u, v}} + \beta_i^{\perp} \norm{P_{u^{\perp}} v}\right) \1\set{\beta_i \geq 0, \beta_i^{\perp} \geq 0}.
		\label{eq:aiming-nonnegative-correlation}
	\end{align}

	\paragraph{Case 2: $\ip{u, v} < 0$.}
	In this case, we may lower bound the sum by the following expression:
	\begin{align}
		 & \ip{\partial f(x), x - \sol}                                                                                   \notag         \\ & \geq
		\frac{\norm{x - \sol}}{m} \sum_{ i = 1 }^m
		\exp(-3 \beta_i \norm{\sol})
		\abs{\beta_i \ip{u, v} + \beta_i^{\perp} \norm{P_{u^{\perp}} v}} \1\set{\beta_i \geq 0, \beta_i^{\perp} \leq 0}           \notag \\
		 & =
		\frac{\norm{x - \sol}}{m}
		\sum_{i = 1}^m \exp(-3 \beta_i \norm{\sol})
		\left(\beta_i (-\ip{u, v}) - \beta_i^{\perp} \norm{P_{u^{\perp}} v}\right) \1\set{\beta_i \geq 0, \beta_i^{\perp} \leq 0} \notag \\
		 & \overset{\mathclap{(d)}}{=}
		\frac{\norm{x - \sol}}{m}
		\sum_{i = 1}^m \exp(-3 \beta_i \norm{\sol})
		\left( \beta_i \abs{\ip{u, v}} + \widetilde{\beta}^{\perp}_i \norm{P_{u^{\perp}} v} \right)
		\1\set{\beta_i, \widetilde{\beta}^{\perp}_i \geq 0},
		\label{eq:aiming-nonpositive-correlation}
	\end{align}
	where $\widetilde{\beta}_i \sim \cN(0, 1) \indep \beta_i$,
	using the fact that the expression inside the absolute value is nonpositive.
	Since the sum in~\eqref{eq:aiming-nonpositive-correlation} is distributionally identical
	to~\eqref{eq:aiming-nonnegative-correlation}, it suffices to study
	\begin{equation}
		(\natural) := \frac{1}{m}
		\sum_{i = 1}^m \exp(-3\beta_i \norm{\sol})
		\left(\beta_i \abs{\ip{u, v}} + \beta_i^{\perp} \norm{P_{u^{\perp}} v}\right)
		\1\set{\beta_i, \beta_i^{\perp} \geq 0}.
		\label{eq:aiming-joint-lb}
	\end{equation}
	Taking expectations with respect to $\beta$ and $\beta^{\perp}$ and writing $\gamma := 3 \norm{\sol}$ for brevity, we obtain
	\begin{align*}
		 & \mathbb{E}_{\beta, \beta^{\perp}}\left[
			\exp(-\gamma \beta) \left( \beta \abs{\ip{u, v}} + \beta^{\perp} \norm{P_{u^{\perp}} v} \right) 1\set{\beta, \beta^{\perp} \geq 0}
		\right]                                                                                  \\
		 & =
		\abs{\ip{u, v}} \cdot \mathbb{E}_{\beta}\left[
			\exp(-\gamma \beta) \beta_+
			\right] \cdot \prob{\beta^{\perp} \geq 0} +
		\norm{P_{u^{\perp}} v} \mathbb{E}_{\beta^{\perp}}[\beta^{\perp}_+]
		\cdot \mathbb{E}_{\beta}\left[\exp(-\gamma \beta) \1\set{\beta \geq 0}\right]            \\
		 & =
		\frac{\abs{\ip{u, v}}}{4}
		\left(
		\sqrt{\frac{2}{\pi}} - \gamma \exp\left(\frac{\gamma^2}{2}\right) \erfc\left(
		\frac{\gamma}{\sqrt{2}}
		\right)
		\right)
		+ \frac{\norm{P_{u^{\perp}} v}}{4}
		\sqrt{\frac{2}{\pi}} \exp\left(\frac{\gamma^2}{2}\right) \erfc\left(
		\frac{\gamma}{\sqrt{2}}
		\right)                                                                                  \\
		 & \geq
		\frac{\abs{\ip{u, v}}}{4(1 + \pi \gamma^2)} \cdot \sqrt{\frac{2}{\pi}} +
		\frac{\norm{P_{u^{\perp}} v}}{4} \sqrt{\frac{2}{\pi}} \exp\left(\frac{\gamma^2}{2}\right)
		\erfc\left(\frac{\gamma}{\sqrt{2}}\right)                                                \\
		 & =
		\frac{\abs{\ip{u, v}}}{4(1 + \pi \gamma^2)} \cdot \sqrt{\frac{2}{\pi}} +
		\frac{\norm{P_{u^{\perp}} v}}{4\gamma} \sqrt{\frac{2}{\pi}} \cdot \gamma \exp\left(\frac{\gamma^2}{2}\right)
		\erfc\left(\frac{\gamma}{\sqrt{2}}\right)                                                \\
		 & \geq
		\sqrt{\frac{1}{8 \pi}} \cdot \min\set{\frac{1}{1 + \pi \gamma^2}, \frac{1}{2 \gamma}}
		\left(
		\abs{\ip{u, v}} + \norm{P_{u^{\perp}} v}
		\right)                                                                                  \\
		 & \geq
		\sqrt{\frac{1}{8 \pi}} \cdot \frac{1}{1 + \pi \gamma^2} \norm{P_{u} v + P_{u^{\perp}} v} \\
		 & = \sqrt{\frac{1}{8 \pi}} \frac{1}{1 + \pi \gamma^2},
	\end{align*}
	using~\cref{lemma:exponential-gaussian-integrals} in the second equality, \cref{lemma:xerfcx-ub} in the first inequality, \cref{lemma:xerfcx-increasing}
	and the bound $\gamma \geq 1$ in the second inequality, the fact that $\abs{\ip{u, v}} = \norm{P_{u} v}$
	and the triangle inequality in the last inequality, and $v \in \mathbb{S}^{d-1}$ in the last step.
	As a result, we have
	\begin{equation*}
		\mathbb{E}[\ip{\partial f(x), x - \sol}] \geq
		\frac{\norm{x - \sol}}{\sqrt{8 \pi}\left(1 + 9 \pi \norm{\sol}^2\right)},
		\quad \text{for any $x \in \cB(\bm{0}; 3 \norm{\sol})$.}
	\end{equation*}
	This is precisely the claimed lower bound, and thus completes the proof.
\end{proof}

\subsubsection{Proof of~\cref{prop:aiming-unif-lb}}
\begin{proof}
	Our proof relies on the generic chaining technique due to Talagrand~\cite{Talagrand14}, using
	the refinements of~\cite{Dirksen2015}. Setting the stage, consider a random process indexed
	by a set $\cT$, $(Z_t)_{t \in \cT}$. We say that $(Z_t)$ has \emph{mixed tails} with respect
	to metrics $(d_1, d_2)$ if
	\begin{equation}
		\prob{\abs{Z_t - Z_{s}} \geq u d_1(s, t) + \sqrt{u} d_2(s, t)} \leq e^{-u}.
		\label{eq:mixed-tail}
	\end{equation}
	Under this condition, \cite[Theorem 3.5]{Dirksen2015} furnishes the following uniform bound:
	for any $t_0 \in \cT$, it holds that
	\begin{equation}
		\prob{\sup_{t \in \cT} \abs{Z_t - Z_{t_0}} \geq C (\gamma_1(\cT, d_1) + \gamma_2(\cT, d_2))
			+ u \, \diam_{d_1}(\cT) + \sqrt{u} \, \diam_{d_2}(\cT)} \leq e^{-u},
		\label{eq:uniform-mixed-tail}
	\end{equation}
	where $\diam_{d}(\cT) := \sup_{s, t \in \cT} d(s, t)$ and $\gamma_{\alpha}(\cT, d)$
	is the Talagrand's \emph{$\gamma$-functional}:
	\[
		\gamma_{\alpha}(\cT, d) = \inf_{(\cT_k)} \sup_{t \in \cT} \sum_{k = 0}^{\infty} 2^{\sfrac{k}{\alpha}} d(t, \cT_k), \qquad
		\cT_k \in \set{\cH \subset \cT \mid \abs{\cH} \leq 2^{2^k}}.
	\]
	The first step in our proof is to verify condition~\eqref{eq:mixed-tail} for the random process
	at hand.
	\begin{claim}[Mixed tail condition]
		\label{claim:mixed-tail}
		For any $(u, v) \in \Sbb^{d-1} \times \Sbb^{d-1}$ and $i \in [m]$, define the process
		\begin{equation}
			Z_{i}(u, v) := \expfun{-3\ip{a_i, u} \norm{\sol}} \abs{\ip{a_i, v}} \1\set{\ip{a_i, u} > 0}
			+ \strictind\set{\ip{a_i, u} = 0} \ip{a_i, \sol}_{-}.
		\end{equation}
		Then the process $Z(u, v) := \frac{1}{m} \sum_{i = 1}^m Z_{i}(u, v) - \expec{Z_i(u, v)}$ satisfies~\eqref{eq:mixed-tail} with
		\[
			d_1 = \frac{(1 + 3 \norm{\sol}) \deuc}{m}, \quad d_2 = \frac{(1 + 3 \norm{\sol}) \deuc}{\sqrt{m}},
		\]
		where $\deuc$ denotes the Euclidean metric on $\Rbb^{d} \times \Rbb^{d}$.
	\end{claim}
	\begin{proof}[Proof of~\cref{claim:mixed-tail}]
		Let $(u, v)$ and $(\bar{u}, \bar{v}) \in \Sbb^{d-1} \times \Sbb^{d-1}$ and
		consider the difference $Z_{i}(u,v) - Z_{i}(\bar{u}, \bar{v})$. Writing $\gamma := 3 \norm{\sol}$, we obtain
		\begin{align*}
			\norm{Z_{i}(u, v) - Z_{i}(\bar{u}, \bar{v})}_{\psi_1} & \leq
			\begin{aligned}[t]
				 & \norm{(\abs{\ip{a_i, v}} - \abs{\ip{a_i, \bar{v}}}) e^{-\ip{a_i, u} \gamma} \1\set{\ip{a_i, u} \geq 0}}_{\psi_1}                                                \\
				 & + \norm{(e^{-\ip{a_i, u} \gamma} \1\set{\ip{a_i, u} \geq 0} - e^{-\ip{a_i, \bar{u}} \gamma} \1\set{\ip{a_i, \bar{u}} \geq 0}) \abs{\ip{a_i, \bar{v}}}}_{\psi_1} \\
				 & + \norm{\ip{a_i, \sol}_{-} \left( \strictind\set{\ip{a_i, u} = 0 - \strictind\set{\ip{a_i, \bar{u}} = 0}} \right)}_{\psi_1}
			\end{aligned} \\
			                                                      & \leq
			\norm{\abs{\ip{a_i, v - \bar{v}}}}_{\psi_1} + \gamma \norm{\abs{\ip{a_i, u - \bar{u}}} \abs{\ip{a, \bar{v}}}}_{\psi_1}                                                         \\
			                                                      & \leq
			\norm{v - \bar{v}} + \norm{u - \bar{u}} \cdot \gamma                                                                                                                           \\
			                                                      & \lesssim
			(1 + \gamma) \norm{\bmx{u                                                                                                                                                      \\ v} - \bmx{\bar{u} \\ \bar{v}}},
		\end{align*}
		using the fact that the mapping $x \mapsto e^{-\gamma x} \1\set{x \geq 0}$ is $\gamma$-Lipschitz on $x \geq 0$, as well as the property
		$\norm{X Y}_{\psi_1} \leq \norm{X}_{\psi_2} \norm{Y}_{\psi_2}$~\cite[Lemma 2.7.7]{Vershynin18}; in particular,
		using the latter property implies that the last term in the decomposition satisfies
		\begin{align*}
			\norm{\ip{a_i, \sol}_{-} \left(\strictind\set{\ip{a_i, u} = 0} - \strictind\set{\ip{a_i, \bar{u}} = 0}\right)}_{\psi_1} & \lesssim
			\norm{\strictind\set{\ip{a_i, u} = 0}}_{\psi_2} +
			\norm{\strictind\set{\ip{a_i, \bar{u}} = 0}}_{\psi_2}
			\\
			                                                                                                                        & = 0,
		\end{align*}
		where the last line follows from~\cite[Proposition 2.5.2]{Vershynin18} and the identity
		\begin{align*}
			\norm{\strictind\set{\ip{a_i, u} = 0}}_{L^p} & =
			\left( \expec{(\strictind\set{\ip{a_i, u} = 0})^p} \right)^{\frac{1}{p}} \\
			                                             & =
			\left( \expec{\strictind \set{\ip{a_i, u} = 0}} \right)^{\frac{1}{p}}    \\
			                                             & = 0.
		\end{align*}
		From the Bernstein inequality~\cite[Corollary 2.8.3]{Vershynin18}, it follows that
		\begin{align}
			\prob{\abs{Z(u, v) - Z(\bar{u}, \bar{v})} \geq t} & \leq 2e^{-c \min\set{\frac{mt^2}{(1 + \gamma)^2 \deuc^2((u, v), (\bar{u}, \bar{v}))}, \frac{mt}{(1 + \gamma)\deuc((u, v), (\bar{u}, \bar{v}))}}}, \label{eq:berstein} \\
			\text{where} \quad Z(u, v)                        & := \frac{1}{m} \sum_{i = 1}^m Z_{i}(u, v) - \mathbb{E}[Z_{i}(u, v)].
			\notag
		\end{align}
		To argue that $Z(u, v)$ has mixed tails, we let
		\[
			d_1 = \frac{(1 + \gamma) \deuc}{m}, \;\; d_2 = \frac{(1 + \gamma) \deuc}{\sqrt{m}}, \;\;
			s = t \cdot d_1((u, v), (\bar{u}, \bar{v})) + \sqrt{t} \cdot d_2((u, v), (\bar{u}, \bar{v})).
		\]
		Substituting in~\eqref{eq:berstein}, and noting $s \geq t d_1((u, v), (\bar{u}, \bar{v}))$ and
		$s^2 \geq t d_2^2((u, v), (\bar{u}, \bar{v}))$, we obtain
		\begin{equation*}
			\prob{\abs{Z(u, v) - Z(\bar{u}, \bar{v})} \geq s} \leq 2\expfun{
				-c \min\set{\frac{s^2}{d_2^2}, \frac{s}{d_1}}
			} \leq 2e^{-c t}.
		\end{equation*}
		Relabeling and adjusting constants yields the mixed tail condition.
	\end{proof}
	With~\cref{claim:mixed-tail} at hand, we can invoke~\eqref{eq:uniform-mixed-tail} for
	$\cT := (\Sbb^{d-1} \times \Sbb^{d-1}) \cup \set{\bm{0}}$:
	\begin{equation}
		\prob{
			\sup_{u, v} \abs{Z(u, v)} \geq
			C (\gamma_1(\cT, d_1) + \gamma_2(\cT, d_2)) +
			t \, \diam_{d_1}(\cT) + \sqrt{t}\, \diam_{d_2}(\cT)
		} \leq 2e^{-t}.
		\label{eq:uniform-mixed-tail-ii}
	\end{equation}
	To simplify~\eqref{eq:uniform-mixed-tail-ii}, we bound the $\gamma$-functionals using
	Dudley's entropy integral method.

	\begin{claim}[Dudley bounds I]
		\label{claim:bound-gamma-1}
		Let $\cT := (\Sbb^{d-1} \times \Sbb^{d-1}) \cup \set{\bm{0}}$. We have
		\begin{equation}
			\diam_{d_1}(\cT) \lesssim \frac{(1 + 3 \norm{\sol})}{m}, \;\;
			\gamma_1(\cT, d_1) \lesssim \frac{(1 + 3 \norm{\sol}) d}{m}.
		\end{equation}
	\end{claim}
	\begin{proof}[Proof of~\cref{claim:bound-gamma-1}]
		Using Dudley's entropy integral method yields
		\begin{align*}
			\gamma_1(\cT, d_1) & = \frac{1 + \gamma}{m} \gamma_1(\cT, \norm{\cdot})                                       \\
			                   & \lesssim \frac{1 + \gamma}{m} \int_{0}^{\infty} \log_+ \cN\left(\cT, u\right) \dd{u}     \\
			                   & = \frac{1 + \gamma}{m} \int_0^{\infty} \log_+ \cN(\cT, u) \dd{u}                         \\
			                   & = \frac{1 + \gamma}{m} \int_0^{1} \log \cN(\cT, u) \dd{u}                                \\
			                   & \lesssim \frac{(1 + \gamma) d}{m} \int_0^{1} \log \frac{1}{\varepsilon} \dd{\varepsilon} \\
			                   & = \frac{(1 + 3 \norm{\sol}) d}{m},
		\end{align*}
		using the fact that the $\epsilon$-covering number of the unit sphere is $(1 / \epsilon)^d$,
		thus the covering number of their Cartesian product is $(1 / \epsilon)^{2d}$. At the same time,
		\begin{align*}
			\diam_{d_1}(\cT) & = \sup_{s, t \in \cT} d_1(s, t)                               \\
			                 & = \frac{1 + 3 \norm{\sol}}{m} \sup_{s, t \in \cT} \deuc(s, t) \\
			                 & \leq \frac{1 + 3 \norm{\sol}}{m} \cdot
			\sqrt{
				\sup_{u, \bar{u} \in \Sbb^{d-1}} \norm{u - \bar{u}}^2 +
				\sup_{v, \bar{v} \in \Sbb^{d-1}} \norm{v - \bar{v}}^2
			}                                                                                \\
			                 & \leq
			\frac{2 \sqrt{2}(1 + 3 \norm{\sol})}{m},
		\end{align*}
		as expected. This completes the proof.
	\end{proof}

	\begin{claim}[Dudley bounds II]
		\label{claim:bound-gamma-2}
		Let $\cT := (\Sbb^{d-1} \times \Sbb^{d-1}) \cup \set{\bm{0}}$. We have
		\begin{equation}
			\diam_{d_2}(\cT) \lesssim \frac{(1 + 3 \norm{\sol})}{\sqrt{m}}, \;\;
			\gamma_2(\cT, d_2) \lesssim (1 + 3 \norm{\sol}) \sqrt{\frac{d}{m}}.
		\end{equation}
	\end{claim}
	\begin{proof}[Proof of~\cref{claim:bound-gamma-2}]
		Proceeding as in the proof of~\cref{claim:bound-gamma-1}, we obtain
		\begin{align*}
			\gamma_2(\cT, d_2) & = \frac{1 + \gamma}{\sqrt{m}} \gamma_2(\cT, \norm{\cdot})                                             \\
			                   & \lesssim \frac{1 + \gamma}{\sqrt{m}} \int_{0}^{\infty} \sqrt{\log \cN\left(\cT, u\right)} \dd{u}      \\
			                   & = \frac{1 + \gamma}{\sqrt{m}} \int_0^{\mathrm{diam}(\cT)} \sqrt{\log\cN(\cT, u)} \dd{u}               \\
			                   & \lesssim \frac{(1 + \gamma)\sqrt{d}}{m} \int_0^{1} \sqrt{\log \frac{1}{\varepsilon}} \dd{\varepsilon} \\
			                   & = (1 + 3 \norm{\sol}) \sqrt{\frac{d}{m}}.
		\end{align*}
		Similarly, to control the diameter under $d_2$ we obtain
		\begin{align*}
			\diam_{d_2}(\cT) & = \sup_{s, t \in \cT} d_2(s, t)                                      \\
			                 & = \frac{1 + 3 \norm{\sol}}{\sqrt{m}} \sup_{s, t \in \cT} \deuc(s, t) \\
			                 & \leq \frac{1 + 3 \norm{\sol}}{\sqrt{m}} \cdot
			\sqrt{
				\sup_{u, \bar{u} \in \Sbb^{d-1}} \norm{u - \bar{u}}^2 +
				\sup_{v, \bar{v} \in \Sbb^{d-1}} \norm{v - \bar{v}}^2
			}                                                                                       \\
			                 & \leq
			\frac{2 \sqrt{2}(1 + 3 \norm{\sol})}{\sqrt{m}},
		\end{align*}
		as expected. This completes the proof.
	\end{proof}
	Combining~\cref{eq:uniform-mixed-tail-ii,claim:bound-gamma-1,claim:bound-gamma-2}, we obtain
	\begin{align*}
		 & \prob{\sup_{u, v} \abs{Z(u, v)} \geq C(1 + 3 \norm{\sol}) \left(\sqrt{\frac{d}{m}} + \frac{d}{m}\right)}                                               \\
		 & \leq
		\prob{\sup_{u, v} \abs{Z(u, v)} \geq C \left(\gamma_1(\cT, d_1) + \gamma_2(\cT, d_2)\right) + d \cdot \diam_{d_1}(\cT) + \sqrt{d} \cdot \diam_{d_2}(\cT)} \\
		 & \leq
		2e^{-d}.
	\end{align*}
	By definition of $Z(u, v)$ and~\cref{lemma:aiming-lb}, it follows that
	\[
		\frac{1}{m} \sum_{i = 1}^m \exp(-3 \ip{a_i, u} \norm{\sol}) \abs{\ip{a_i, v}} \1\set{\ip{a_i, u} \geq 0}
		\geq \frac{1}{\sqrt{8 \pi}\left(1 + 9 \pi \norm{\sol}^2\right)} -
		C \left( \sqrt{\frac{d}{m}} + \frac{d}{m} \right),
	\]
	uniformly over $u, v \in \Sbb^{d-1}$, with probability $1 - 2e^{-d}$. Therefore, we conclude that
	\[
		\min_{v \in \partial f(x)} \ip{v, x - \sol} \geq
		\frac{\norm{x - \sol}}{4 \sqrt{\pi}\left(1 + 9 \pi \norm{\sol}^2\right)}, \;\;
		\forall x \in \cB(\bm{0}; 3 \norm{\sol}) \setD \set{\bm{0}},
	\]
	with probability at least $1 - 2e^{-d}$, as long as $m \gtrsim d \cdot (1 + 9 \pi \norm{\sol}^2)^2 \gtrsim d \norm{\sol}^4$.
\end{proof}

\subsection{Proofs from~\cref{sec:convergence}}

\subsubsection{Proof of~\cref{lemma:Aslow-implies-Bslow}}
\begin{proof}
	We have the following chain of inequalities:
	\begin{align*}
		\norm{x_{t+1} - \sol}^2 & \leq (1 - \rhoslow) \norm{x_t - \sol}^2     \\
		                        & \leq (1 - \rhoslow)^{t} \norm{x_1 - \sol}^2 \\
		                        & \leq (1 - \rhoslow)^t \norm{\sol}^2,
	\end{align*}
	using the definition of $\cA_{\slow}(j)$, \eqref{eq:x1-distance-reminder}, and $x_{0} = \bm{0}$. By the reverse triangle inequality,
	\[
		\norm{x_{t+1}} - \norm{\sol} \leq \norm{x_{t+1} - \sol} \leq (1 - \rhoslow)^{t/2} \norm{\sol}
		\implies
		\norm{x_{t+1}} \leq \norm{\sol}(1 + (1 - \rhoslow)^{t/2}) < 2 \norm{\sol},
	\]
	which proves the upper bound in $\cB_{\slow}(t+1)$. For the lower bound, we have
	\begin{align*}
		\norm{x_t} & = \sup_{u \in \Sbb^{d-1}} \ip{u, x_t}                                                                                         \\
		           & \geq \frac{1}{\norm{\sol}} \ip{\sol, x_t}                                                                                     \\
		           & = \frac{1}{\norm{\sol}} \left( \ip{\sol, \sol} + \ip{\sol, x_t - \sol} \right)                                                \\
		           & \geq \frac{1}{\norm{\sol}} \left(\norm{\sol}^2 - \norm{\sol} \norm{x_t - \sol} \right)                                        \\
		           & \geq \frac{1}{\norm{\sol}} \left(\norm{\sol}^2 - \norm{\sol} \norm{x_1 - \sol}\right)                                         \\
		           & \geq \frac{1}{\norm{\sol}} \norm{\sol}^2 \left(1 - \left(1 - \frac{1}{20 \sqrt{\pi} \norm{\sol}}\right)^{\sfrac{1}{2}}\right) \\
		           & = \frac{1}{40 \sqrt{\pi}},
	\end{align*}
	where the second inequality follows from Cauchy-Schwarz, the third inequality follows
	from $\norm{x_t - \sol} \leq \norm{x_1 - \sol}$, which is implied by the events $\set{\cA_{\slow}(j)}_{j \leq t}$,
	the penultimate inequality follows from~\eqref{eq:x1-distance-reminder}, and the last
	inequality follows from the identity $\sqrt{1 - x} \leq 1 - \frac{x}{2}$.
\end{proof}

\subsubsection{Proof of~\cref{lemma:allslow-implies-Aslow}}
\begin{proof}
	Recall that on the event $\cB_{\slow}(t)$,~\cref{prop:aiming-unif-lb,lemma:sharpness-from-aiming} imply that
	\[
		\begin{aligned}
			\ip{v_t, x_t - \sol} & \geq \mu \norm{x_t - \sol}, \;\; v_t \in \partial f(x_t), \\
			f(x_t)               & \geq \mu \min\set{\norm{x_{t} - \sol}, \norm{x_t}},
		\end{aligned}
	\]
	where $\mu := \frac{1}{4 \sqrt{\pi}(1 + 9 \pi \norm{\sol}^2)}$. At the same time,
	$\set{\cA_{\slow}(j)}_{j < t}$ and~\eqref{eq:x1-distance-reminder} yield
	\[
		\norm{x_{t} - \sol}^2 \leq (1 - \rhoslow)^{t-1} \norm{x_1 - \sol}^2
		\leq \norm{\sol}^2 \left(1 - \frac{1}{20 \sqrt{\pi} \norm{\sol}}\right).
	\]
	As a result,~\cref{lemma:minimum-irrelevant} implies the following lower bound:
	\begin{align}
		f(x_t) \geq \mu \min(\norm{x_t - \sol}, \norm{x_t}) \geq \mu \cdot \min\set{1, \frac{1}{40 \sqrt{\pi} \norm{\sol}}}
		\norm{x_t - \sol}.
		\label{eq:modified-sharpness-fval-only}
	\end{align}
	Proceeding with the convergence analysis, we obtain
	\begin{align*}
		\norm{x_{t+1} - \sol}^2 & = \norm{x_{t} - \eta \frac{f(x_t)}{\norm{v_t}^2} v_t - \sol}^2                             \\
		                        & =
		\norm{x_t - \sol}^2 + \frac{\eta^2 f(x_t)^2}{\norm{v_t}^2} - 2 \frac{\eta f(x_t)}{\norm{v_t}^2} \ip{v_t, x_t - \sol} \\
		                        & =
		\norm{x_t - \sol}^2 -
		\frac{\eta f(x_t)}{\norm{v_t}^2} \left(
		2\ip{v_t, x_t - \sol} - \eta f(x_t)
		\right)                                                                                                              \\
		                        & \leq
		\norm{x_t - \sol}^2 - \frac{\eta f(x_t)}{\norm{v_t}^2}
		\left(
		2 \mu \norm{x_t - \sol} - \mathsf{L} \eta \norm{x_t - \sol}
		\right)                                                                                                              \\
		                        & \leq
		\norm{x_t - \sol}^2 - \frac{\eta f(x_t) \mu}{\lip^2} \norm{x_t - \sol}                                               \\
		                        & \leq
		\norm{x_t - \sol}^2 \left(
		1 - \eta \left( \frac{\mu}{\lip} \right)^2
		\min\set{1, \frac{1}{40 \sqrt{\pi} \norm{\sol}}}
		\right)
		\\
		                        & =
		\norm{x_t - \sol}^2 \left(
		1 - \rhoslow
		\right),
	\end{align*}
	where the first inequality follows from the aiming inequality,
	the penultimate inequality follows from the requirement $\eta \leq \frac{\mu}{\mathsf{L}}$,
	and the last inequality follows from~\eqref{eq:modified-sharpness-fval-only}.
\end{proof}

\subsubsection{Proof of~\cref{lemma:Aslow-until-T0-implies-Bfast}}
\begin{proof}
	Iterating the inequality from the definition of $\cA_{\slow}(j)$, we obtain
	\begin{align*}
		\norm{x_{T_0} - \sol}^2 \leq
		\left(1 - \rhoslow\right)^{T_0} \norm{\sol}^2 <
		\exp(-T_0 \rhoslow) \norm{\sol}^2 \leq \frac{1}{4} \norm{\sol}^2,
	\end{align*}
	where the last step follows from the inequality $T_{0} \rhoslow \geq \log(4)$.
	Consequently,
	\begin{align*}
		\norm{x_{T_0}} & \geq
		\frac{1}{\norm{\sol}} \ip{\sol, x_{T_0}}                                             \\
		               & \geq
		\frac{1}{\norm{\sol}} \left(\norm{\sol}^2 - \norm{\sol} \norm{x_{T_0} - \sol}\right) \\
		               & >
		\frac{1}{\norm{\sol}} \norm{\sol}^2 \left(1 - \frac{1}{2}\right)                     \\
		               & =
		\frac{1}{2} \norm{\sol}.
	\end{align*}
	This proves the lower bound in $\cB_{\fast}(T_0)$; the upper bound follows
	from~\cref{lemma:Aslow-implies-Bslow}.
\end{proof}

\subsubsection{Proof of~\cref{lemma:Allfast-implies-Afast}}
\begin{proof}
	Note that by iterating the definition of $\cA_{\fast}(j)$, we have
	\begin{align*}
		\norm{x_{t} - \sol}^2 & \leq (1 - \rhofast)^{t - T_0} \norm{x_{T_0} - \sol}^2   \\
		                      & \leq (1 - \rhofast)^{t - T_0} \frac{1}{4} \norm{\sol}^2 \\
		                      & < \frac{1}{4} \norm{\sol}^2.
	\end{align*}
	Moreover, $\cB_{\fast}(t)$ implies $\norm{x_t} > \frac{1}{2} \norm{\sol} \geq \norm{x_t - \sol}$.
	Consequently, \cref{lemma:sharpness-from-aiming} yields
	\begin{equation}
		f(x_t) \geq \mu \min(\norm{x_t - \sol}, \norm{x_t}) = \mu \norm{x_t - \sol}.
		\label{eq:sharpness-fval}
	\end{equation}
	Repeating the convergence analysis from the proof of~\cref{lemma:allslow-implies-Aslow}, we obtain
	\begin{align*}
		\norm{x_{t+1} - \sol}^2 & \leq
		\norm{x_t - \sol}^2 - \frac{\eta f(x_t) \mu}{\lip^2} \norm{x_t - \sol} \\
		                        & \leq
		\norm{x_t - \sol}^2 \left(
		1 - \eta \left( \frac{\mu}{\lip} \right)^2
		\right)                                                                \\
		                        & =
		\norm{x_t - \sol}^2 (1 - \rhofast),
	\end{align*}
	where we invoke~\eqref{eq:sharpness-fval} in the last inequality.
\end{proof}

\subsubsection{Proof of~\cref{lemma:Afast-implies-Bfast}}
\begin{proof}
	The analysis mirrors that of~\cref{lemma:Aslow-implies-Bslow}. First, note that
	\begin{align*}
		\norm{x_{t+1} - \sol}^2 & \leq (1 - \rho) \norm{x_t - \sol}^2                   \\
		                        & \leq (1 - \rho)^{t - T_0 + 1} \norm{x_{T_0} - \sol}^2 \\
		                        & < \frac{1}{4} \norm{\sol}^2,
	\end{align*}
	which implies that $\norm{x_{t+1}} > \frac{1}{2} \norm{\sol}$ by a similar argument
	to that used in~\cref{lemma:Aslow-until-T0-implies-Bfast} and
	$\norm{x_{t+1}} < 2 \norm{\sol}$ by repeating the argument from~\cref{lemma:Aslow-implies-Bslow}
	and using $\norm{x_{T_0} - \sol} < \frac{1}{2} \norm{\sol}$.
\end{proof}

\subsubsection{Proof of~\cref{prop:adpolyak-termination}}
\begin{proof}
	Suppose that $\condnum$ is a power of $2$ for simplicity.
	For any $i \geq i_{\star} := \log_{2} (\condnum)$, we have
	\[
		\eta_i := {\condest_i^{-1}} = 2^{-i} \leq \condnum^{-1}.
	\]
	By~\cref{theorem:main}, any call
	to~\cref{alg:polyaksgm} with step-size
	$\eta_i = \condest^{-1}$ running for $T_i$ iterations, where $T_i$ is defined in~\cref{op:iterations} of~\cref{alg:adpolyak},
	will ``succeed'', i.e., return a point $x_i$ satisfying
	\[
		f(x_i) \leq \kappa_i \left(1 - \frac{1}{\kappa_{i}}\right)^{\frac{T_i}{2}} f(x_0)
		\leq \varepsilon f(x_0).
	\]
	Moreover, it is clear that~\cref{alg:adpolyak} will terminate
	as soon as a call to~\cref{alg:polyaksgm} succeeds.
	Therefore, the total number of outer loops in~\cref{alg:adpolyak} is $i_{\star} = \log_{2}(\condnum)$. \\

	To account for the total number of subgradient evaluations, we calculate
	\begin{align*}
		\sum_{j = 0}^{i_{\star}} T_{j} & =
		\sum_{j = 0}^{\log_{2}(\condnum)} 2 \condest_{j}^3 \log \left(\frac{\condest_{j}}{\varepsilon}\right) + 7 \condest_i^{\frac{7}{2}} \\
		                               & =
		2 \cdot \left[ \sum_{j = 0}^{\log_{2}(\condnum)} 2^{3j} \log\left(\frac{2^{j}}{\varepsilon}\right) + 2^{\frac{7j}{2}} \right]      \\
		                               & \leq
		2 \log\left(\frac{\condnum}{\varepsilon}\right) \cdot
		\sum_{j = 0}^{\log_{2}(\condnum)} 8^{j} + 7 \sum_{j = 0}^{\log_2(\condnum)} \left( 2^{\frac{7}{2}} \right)^j                       \\
		                               & =
		2 \log\left(\frac{\condnum}{\varepsilon}\right) \cdot \frac{1}{7} \left(8 \cdot \condnum^3 - 1\right)
		+ 7 \cdot \frac{\left( 2^{\frac{7}{2}} \right)^{\log_2(2 \condnum)} - 1}{2^{\frac{7}{2}} - 1}                                      \\
		                               & \leq
		3 \condnum^3 \log\left(\frac{\condnum}{\varepsilon}\right) + 8 \condnum^{\frac{7}{2}},
	\end{align*}
	using the summation identity for finite geometric series in the penultimate step.
\end{proof}

\section{Technical results}

\begin{lemma}
	\label{lemma:exponential-gaussian-integrals}
	For $X \sim \cN(0, 1)$, we have that
	\begin{subequations}
		\begin{align}
			\expec{e^{-cX} \1\set{X \geq 0}} & =
			\frac{1}{2} \expfun{\frac{c^2}{2}}
			\erfc\left(\frac{c}{\sqrt{2}}\right),
			\label{eq:exp-pos-integral}          \\
			\expec{e^{-cX} X_+}              & =
			\frac{1}{2} \left(
			\sqrt{\frac{2}{\pi}} -
			c \expfun{-\frac{c^2}{2}} \erfc\left(\frac{c}{\sqrt{2}}\right)
			\right).
			\label{eq:exp-max-integral}
		\end{align}
	\end{subequations}
\end{lemma}
\begin{proof}
	The first expectation is the integral
	\begin{align*}
		\expec{e^{-cX} \1\set{X \geq 0}} & =
		\int_{0}^{\infty} \frac{1}{\sqrt{2 \pi}} \expfun{-\frac{x^2}{2} - cx} \dd{x}                                                     \\
		                                 & =
		\int_{0}^{\infty} \frac{1}{\sqrt{2 \pi}} \expfun{-\left(\frac{x^2}{2} + cx + \frac{c^2}{2}\right)} \expfun{\frac{c^2}{2}} \dd{x} \\
		                                 & =
		\expfun{\frac{c^2}{2}} \int_{0}^{\infty} \frac{1}{\sqrt{2 \pi}} \expfun{-\frac{(x + c)^2}{2}} \dd{x}                             \\
		                                 & =
		\expfun{\frac{c^2}{2}} \int_{\frac{c}{\sqrt{2}}}^{\infty} \frac{1}{\sqrt{\pi}} \expfun{-z^2} \dd{z}
		\qquad (z \gets \frac{x + c}{\sqrt{2}})                                                                                          \\
		                                 & =
		\expfun{\frac{c^2}{2}} \frac{1}{2} \cdot \frac{2}{\sqrt{\pi}} \int_{\frac{c}{\sqrt{2}}}^{\infty} \expfun{-z^2} \dd{z}            \\
		                                 & =
		\frac{1}{2} \expfun{\frac{c^2}{2}} \erfc\left(\frac{c}{\sqrt{2}}\right),
	\end{align*}
	using~\eqref{eq:erfc-defn} in the last equality; this proves~\eqref{eq:exp-pos-integral}. The second expectation is
	\begin{align*}
		\expec{e^{-cX} X_+} & =
		\int_{0}^{\infty} \frac{x}{\sqrt{2 \pi}} \expfun{-\frac{x^2}{2} - cx} \dd{x}           \\
		                    & =
		\int_{0}^{\infty} \frac{(x + c)}{\sqrt{2 \pi}} \expfun{-\frac{x^2}{2} - cx} \dd{x}
		- c \cdot \int_{0}^{\infty} \frac{1}{\sqrt{2 \pi}} \expfun{-\frac{x^2}{2} - cx} \dd{x} \\
		                    & =
		\expfun{\frac{c^2}{2}}
		\int_{\frac{c}{\sqrt{2}}}^{\infty} z \cdot \sqrt{\frac{2}{\pi}} \expfun{-z^2} \dd{z}
		- \frac{c}{2} \expfun{\frac{c^2}{2}} \erfc\left(\frac{c}{\sqrt{2}}\right),
	\end{align*}
	recognizing the integral from~\eqref{eq:exp-pos-integral} in the second term. Note that
	\[
		\frac{\dd{}}{\dd{z}} e^{-z^2} =
		-2ze^{-z^2} \implies
		\int_{0}^{\infty} z \cdot \sqrt{\frac{2}{\pi}} e^{-z^2} \dd{z} =
		-\sqrt{\frac{2}{\pi}} \frac{1}{2} \left\{
		e^{-z^2}
		\right\}_{\frac{c}{\sqrt{2}}}^{\infty} =
		\frac{1}{2} \expfun{-\frac{c^2}{2}} \sqrt{\frac{2}{\pi}},
	\]
	cancelling out the leading $\exp(c^2 / 2)$; this proves~\eqref{eq:exp-max-integral}.
\end{proof}

\begin{lemma}
	\label{lemma:erfcx-increasing}
	The function $f(x) := \exp(x^2) \cdot \erfc(x)$ is monotone decreasing for $x \geq 0$.
\end{lemma}
\begin{proof}
	The first derivative of $f$ is equal to
	\[
		f'(x) = 2x \cdot \exp(x^2) \erfc(x) - \frac{2}{\sqrt{\pi}}.
	\]
	We now use the inequality $\erfc(x) \leq \frac{2 e^{-x^2}}{\sqrt{\pi}(x + \sqrt{x^2 + 4/\pi})}$,
	which yields
	\begin{align*}
		f'(x) & \leq 2x \cdot \frac{2}{\sqrt{\pi}} \cdot \frac{1}{x + \sqrt{x^2 + 4/\pi}} - \frac{2}{\sqrt{\pi}} \\
		      & <
		\frac{4x}{\sqrt{\pi}} \cdot \frac{1}{2 x} - \frac{2}{\sqrt{\pi}}                                         \\
		      & \leq 0,
	\end{align*}
	which completes the proof of the claim.
\end{proof}

\begin{lemma}
	\label{lemma:xerfcx-increasing}
	The function $\varphi(x) := x \exp(x^2) \cdot \erfc(x)$ is monotone increasing for $x \geq 0$.
\end{lemma}
\begin{proof}
	The first derivative of $\varphi$ is
	\begin{align*}
		\varphi'(x) & = \exp(x^2) (2x^2 + 1) \erfc(x) - \frac{2x}{\sqrt{\pi}}
	\end{align*}
	Clearly, $\varphi'(0) > 0$. Now for any $x > 0$, we have
	\begin{align*}
		e^{x^2} \erfc(x) & \geq \frac{2}{\sqrt{\pi}(x + \sqrt{x^2 + 2})}          \\
		                 & =
		\frac{2}{\sqrt{\pi}} \frac{1}{x\left(1 + \sqrt{1 + \frac{2}{x^2}}\right)} \\
		                 & \overset{(\sharp)}{>}
		\frac{2}{x \sqrt{\pi}} \frac{1}{2 + \frac{1}{x^2}}                        \\
		                 & =
		\frac{2}{\sqrt{\pi}} \frac{1}{2x + \frac{1}{x}}                           \\
		                 & =
		\frac{2x}{\sqrt{\pi}} \frac{1}{2x^2 + 1},
	\end{align*}
	where $(\sharp)$ follows from $\sqrt{1 + x} < 1 + \frac{x}{2}$, valid
	for any $x > 0$. Finally, multiplying with $(2x^2 + 1)$
	and subtracting $2x / \sqrt{\pi}$ yields $\varphi'(x) > 0$ for any $x > 0$,
	completing the proof.
\end{proof}

\begin{lemma}
	\label{lemma:xerfcx-ub}
	For any $x \geq 1$, the following bound holds:
	\begin{equation}
		x \expfun{\frac{x^2}{2}} \erfc\left(\frac{x}{\sqrt{2}}\right) \leq
		\sqrt{\frac{2}{\pi}} \left(
		1 - \frac{1}{1 + \pi x^2}
		\right).
		\label{eq:xerfcx-ub}
	\end{equation}
\end{lemma}
\begin{proof}
	Starting from the known inequality
	\[
		\erfc(x) \leq \frac{2 e^{-x^2}}{\sqrt{\pi}(x + \sqrt{x^2 + \frac{4}{\pi}})},
	\]
	we successively obtain
	\begin{align*}
		x \exp(x^2 / 2) \erfc(x / \sqrt{2}) & \leq
		\frac{2x}{\sqrt{\pi}} \frac{1}{\frac{x}{\sqrt{2}} + \sqrt{\frac{x^2}{2} + \frac{4}{\pi}}} \\
		                                    & =
		\frac{2 \sqrt{2}}{\sqrt{\pi}} \frac{1}{1 + \sqrt{1 + \frac{8}{\pi x^2}}}                  \\
		                                    & \leq
		\frac{2 \sqrt{2}}{\sqrt{\pi}} \frac{1}{2 + \frac{2}{\pi x^2}}                             \\
		                                    & =
		\sqrt{\frac{2}{\pi}} \frac{1}{1 + \frac{1}{\pi x^2}},
	\end{align*}
	where the last inequality is due to the fact that
	\[
		\sqrt{1 + y^2} \geq 1 + \frac{y^2}{4}, \quad \forall y \in [0, 1].
	\]
	Finally, we rewrite the last fraction as
	\[
		\frac{1}{1 + \frac{1}{\pi x^2}} =
		\frac{1 + \frac{1}{\pi x^2}}{1 + \frac{1}{\pi x^2}} -
		\frac{\frac{1}{\pi x^2}}{1 + \frac{1}{\pi x^2}}
		= 1 - \frac{1}{\pi x^2 + 1}.
	\]
	This completes the proof of~\eqref{eq:xerfcx-ub}.
\end{proof}

\begin{lemma}
	\label{lemma:minimum-irrelevant}
	Suppose that $x$ satisfies $\norm{x - \sol}^2 < (1 - \delta) \norm{\sol}^2$. Then
	\begin{align*}
		\min\left(\norm{x - \sol}, \norm{x}, \dist(x, \cB^c(\bm{0}; 3 \norm{\sol}))\right)
		\geq \min\set{1, \frac{\delta}{2 \sqrt{1 - \delta}}} \cdot \norm{x - \sol}.
	\end{align*}
\end{lemma}
\begin{proof}
	We have the following sequence of inequalities:
	\begin{align*}
		\dist(x, \cB^c(\bm{0}; 3 \norm{\sol})) & =
		\norm{x - \proj_{\cB^c(\bm{0}; 3 \norm{\sol})}(x)}               \\
		                                       & =
		\norm{x - \sol + \sol - \proj_{\cB^c(\bm{0}; 3 \norm{\sol})}(x)} \\
		                                       & \geq
		\norm{\sol - \proj_{\cB^c(\bm{0}; 3 \norm{\sol})}(x)}
		- \norm{x - \sol}                                                \\
		                                       & \geq
		\norm{\sol - \proj_{\cB^c(\bm{0}; 3 \norm{\sol})}(\sol)}
		- \norm{x - \sol}                                                \\
		                                       & >
		\norm{\sol - \proj_{\cB^c(\bm{0}; 3 \norm{\sol})}(\sol)}
		- \norm{\sol}                                                    \\
		                                       & \geq
		2 \norm{\sol} - \norm{\sol}                                      \\
		                                       & >
		\norm{x - \sol},
	\end{align*}
	which follow from $\norm{x - \proj_{C}(y)} \geq \norm{x - \proj_{C}(x)}$, our
	assumptions on $\norm{x - \sol}$, and
	\[
		\norm{\sol - \proj_{\cB^c(0; 3 \norm{\sol})}(\sol)}
		\geq
		\norm{\proj_{\cB^c(0; 3 \norm{\sol})}(\sol)} - \norm{\sol} \geq
		3 \norm{\sol} - \norm{\sol} = 2 \norm{\sol}.
	\]
	This shows that $\norm{x - \sol} < \dist(x, \cB^c(\bm{0}; 3 \norm{\sol}))$. At the same time,
	\begin{align*}
		(1 - \delta) \norm{\sol}^2       & \geq \norm{x - \sol}^2                        \\
		                                 & = \norm{x}^2 + \norm{\sol}^2 - 2 \ip{x, \sol} \\
		                                 & \geq \norm{\sol}^2 - 2 \ip{x, \sol}           \\
		\Leftrightarrow
		\ip{x, \frac{\sol}{\norm{\sol}}} & \geq \frac{\delta}{2} \norm{\sol}.
	\end{align*}
	From the previous display and our assumption $\norm{x - \sol} \leq \sqrt{1 - \delta} \norm{\sol}$, we deduce
	\begin{align*}
		\norm{x} & = \sup_{u \in \Sbb^{d-1}} \ip{x, u}                                          \\
		         & \geq \ip{x, \frac{\sol}{\norm{\sol}}}                                        \\
		         & \geq \frac{\delta}{2} \norm{\sol}                                            \\
		         & = \frac{\delta}{2} \frac{\norm{\sol}}{\norm{x - \sol}} \cdot \norm{x - \sol} \\
		         & \geq \frac{\delta}{2 \sqrt{1 - \delta}} \norm{x - \sol},
	\end{align*}
	which completes the proof of the claim.
\end{proof}

\section{Projecting to the TV norm ball}
\label{appendix:sec:tv-norm-ball-projection}
In this section, we describe an iterative approach based on
the Douglas-Rachford splitting method \citep{DR56} to compute
the orthogonal projection onto the TV ball. First, we rewrite
\begin{align*}
	\proj_{\cX}(y) & = \argmin_{x} \set{\frac{1}{2} \norm{x - y}^2 \mid \tvnorm{\mat(x)} \leq \lambda} \\
	               & = \argmin_{x}
	\set{\frac{1}{2} \norm{x - y}^2 \mid \norm{z}_{n^2, 2} \leq \lambda, \ J_{\mathsf{TV}}(x) = z}     \\
	               & = \argmin_{x, z}
	\set{\frac{1}{2} \norm{x - y}^2 + \delta_{\set{\norm{\cdot}_{n^2, 2} \leq \lambda}}(z)
	\mid J_{\mathsf{TV}}(x) = z}                                                                       \\
	               & =
	\argmin_{x, z}
	\set{
		\frac{1}{2} \norm{x - z}^2 + \delta_{(\Rbb^{n^2}) \times \set{\norm{\cdot}_{n^2, 2} \leq \lambda}}(x, z)
		+ \delta_{\gph(J_{\mathsf{TV}})}(x, z)
	}
\end{align*}
where $\delta_{\cC}(x) = 0$ for $x \in \cC$ and $+\infty$ otherwise. The operator $J_{\mathsf{TV}}: \Rbb^{n^2} \to \Rbb^{2n^2}$
is given by
\begin{subequations}
	\begin{align}
		[J_{\mathsf{TV}}(x)]_{(i - 1) \cdot n + j}       & = \mathbf{M}(x)_{i+1, j} - \mathbf{M}(x)_{i, j}, \quad i, j = 1, \dots, n-1
		\label{eq:dtv-def-1}                                                                                                           \\
		[J_{\mathsf{TV}}(x)]_{n^2 + (i - 1) \cdot n + j} & =
		\mathbf{M}(x)_{i, j+1} - \mathbf{M}(x)_{i, j}, \quad i, j = 1, \dots, n-1
		\label{eq:dtv-def-2}
	\end{align}
\end{subequations}
and $\norm{z}_{k, 2}: \Rbb^{2k} \to \Rbb^{k}$ is the following group norm:
\begin{equation}
	\norm{z}_{k, 2} =
	\sum_{i = 1}^{k} \norm{(Z_{i}, Z_{k + i})}.
	\label{eq:partial-norm}
\end{equation}
With this in hand, we can now write
\begin{subequations}
	\begin{align}
		\proj_{\cX}(y) & = \argmin_{x, z} \set{
			h_1(x, z) + h_2(x, z)
		},                                                                                                                                   \label{eq:pre-splitting} \\
		h_{1}(x, z)    & := \frac{1}{2} \norm{x - y}^2 + \delta_{\Rbb^{n^2} \times \set{\norm{\cdot}_{n^2, 2} \leq \lambda}}(x, z) \label{eq:pre-splitting-h1}        \\
		h_{2}(x, z)    & := \delta_{\gph(J_{\mathsf{TV}})}(x, z). \label{eq:pre-splitting-h2}
	\end{align}
\end{subequations}
The formulation in~\eqref{eq:pre-splitting} lends itself to various splitting
algorithms; our implementation uses the Douglas-Rachford splitting method. The
latter maintains an auxiliary sequence $(\bar{x}, \bar{z})$, as well as the
primal iterates $(x, z)$, and alternates between the following steps:
\begin{align*}
	(x_{k+1}, z_{k+1})             & := \prox_{h_1}(\bar{x}_k, \bar{z}_k);                                                                              \\
	(\bar{x}_{k+1}, \bar{z}_{k+1}) & := (\bar{x}_k, \bar{z}_k) + \prox_{h_2}(2 \cdot (x_{k+1}, z_{k+1}) - (\bar{x}_k, \bar{z}_k)) - (x_{k+1}, z_{k+1}).
\end{align*}
We now, we provide low-complexity formulas for the proximal operators
involved when using the Douglas-Rachford splitting method to compute a projection
onto the norm ball $\tvnorm{\mat(x)} \leq \lambda$. Throughout, we fix
\begin{subequations}
	\begin{align}
		h_{1}: \Rbb^{n^2 \times {2n^2}} \to \Rbb, \quad & h_{1}(x, z) :=
		\frac{1}{2} \norm{x - y}^2 + \delta_{\Rbb^{n^2} \times \set{\norm{\cdot}_{n^2, 2} \leq \lambda}}(x, z), \\
		h_2: \Rbb^{n^2 \times 2n^2} \to \Rbb, \quad     & h_{2}(x, z) :=
		\delta_{\gph(J_{\mathsf{TV}})}(x, z),
	\end{align}
\end{subequations}
where $\norm{z}_{k, 2}: \Rbb^{k} \to \Rbb$ is the following group norm:
\[
	\norm{z}_{k, 2} := \sum_{i = 1}^{k} \norm{\pmx{z_{i} & z_{k + i}}}.
\]

\begin{lemma}
	We have the following formulas:
	\begin{subequations}
		\begin{align}
			\prox_{h_1}(x, z) & = \begin{bmatrix}
				                      \frac{x + y}{2} \\
				                      \proj_{\norm{\cdot}_{n^2, 2} \leq \lambda}(z)
			                      \end{bmatrix}, \label{eq:prox-h1} \\
			\prox_{h_2}(x, z) & = \proj_{\gph(J_{\mathsf{TV}})}(x, z).
			\label{eq:prox-h2}
		\end{align}
	\end{subequations}
\end{lemma}
\begin{proof}
	We start from the following characterization:
	\[
		(x_+, z_+) = \prox_{h_1}(x, z) \Leftrightarrow
		(x - x_+, z - z_+) \in \partial h_1(x_+, z_+).
	\]
	To compute $\partial h_{1}$, we use the following properties (see, e.g.,~\cite{Rusz06}):
	\begin{itemize}
		\item $\partial \delta_{\cC}(\bar{x}) = \cN_{\cC}(\bar{x})$, where $\cN_{\cC}(\bar{x})$ is the normal cone of $\cC$ at $\bar{x}$;
		\item $\cN_{\cC_1 \times \cC_2}(\bar{x}, \bar{z}) = \cN_{\cC_1}(\bar{x}) \times \cN_{\cC_2}(\bar{z})$ for any convex $\cC_1, \cC_2$.
	\end{itemize}
	Indeed, from the above it follows that
	\begin{align*}
		\partial h_{1}(x_+, z_+) & = (x_+ - y, \bm{0}) + \cN_{\Rbb^{n^2} \times \set{\norm{\cdot}_{n^2, 2} \leq \lambda}}(x_+, z_+)       \\
		                         & = (x_+ - y, \bm{0}) + \cN_{\Rbb^{n^2}}(x_+) \times \cN_{\set{\norm{\cdot}_{n^2, 2} \leq \lambda}}(z_+) \\
		                         & = (x_+ - y, \bm{0}) + \set{\bm{0}_{n^2}} \times \cN_{\set{\norm{\cdot}_{n^2, 2} \leq \lambda}}(z_+)
	\end{align*}
	In particular, since this system is separable, we obtain
	\begin{align*}
		x - x_+ & = (x_+ - y) \implies x_+ = \frac{x + y}{2},                      \\
		z - z_+ & \in \cN_{\set{\norm{\cdot}_{n^2, 2} \leq \lambda}}(z_+) \implies
		z_+ = \proj_{\set{\norm{\cdot}_{n^2, 2} \leq \lambda}}(z).
	\end{align*}
	We defer the derivation of a formula for the projection
	to~\cref{lemma:proj-tvnorm-ball}. Finally, we note
	\begin{align*}
		(x_+, z_+)         = \prox_{h_2}(x, z)
		\Leftrightarrow
		(x_+, z_+)         = \proj_{\gph(J_{\mathsf{TV}})}(x, z).
	\end{align*}
	Since $J_{\mathsf{TV}}$ is a very sparse linear operator, the graph
	projection can be computed very efficiently using a cached sparse Cholesky or
	a sparse \texttt{LDLt} factorization, allowing for repeated linear system
	solves. We defer the details to~\cref{lemma:proj-Jtv-graph}.
\end{proof}

\begin{lemma}
	\label{lemma:proj-tvnorm-ball}
	Let $M \in \Rbb^{m \times n}$ and write $\norm{M}_{1, 2} = \sum_{i = 1}^m \norm{M_{i, :}}$.
	Then we have that
	\[
		\proj_{\set{M \mid \norm{M}_{1, 2} \leq \lambda}}(X) =
		\prox_{\mu_{\star} \norm{M}_{1, 2}}(X),
	\]
	where $\mu_{\star}$ is the unique root of the one-dimensional equation
	\[
		\bignorm{\prox_{\mu \norm{M}_{1, 2}}(X)}_{1,2} - \lambda = 0.
	\]
	In particular, $\mu_{\star}$ can be found in $O(m \log m + mn)$ time.
\end{lemma}
\begin{proof}
	Consider the optimality condition for the projection problem:
	\begin{align*}
		X_{+} = \proj_{\set{U \mid \norm{U}_{1, 2} \leq \lambda}}(X) & \Leftrightarrow
		X - X_{+} \in \cN_{\set{U \mid \norm{U}_{1, 2} \leq \lambda}}(X_{+})
		= \Rbb_{+} (\partial \norm{\cdot}_{1, 2})(X_{+})                               \\
		                                                             & \Leftrightarrow
		\exists \mu > 0: \ (X - X_{+}) \in \mu (\partial \norm{\cdot}_{1, 2})(X_{+}), \;\;
		\norm{X_{+}}_{1, 2} = \lambda                                                  \\
		                                                             & \Leftrightarrow
		\exists \mu > 0: \ \lambda = \bignorm{\prox_{\mu \norm{\cdot}_{1, 2}}(X)}_{1, 2},
	\end{align*}
	using the property that $\cN_{\set{X \mid \norm{X} \leq \lambda}} = \Rbb_{+} (\partial \norm{\cdot})$
	for any norm $\norm{\cdot}$. Examining the resulting one-dimensional equation,
	we note that~\cref{lemma:prox-tvnorm-surrogate} implies that
	\begin{equation}
		\bignorm{\prox_{\mu \norm{\cdot}_{1, 2}}(X)}_{1, 2} = \lambda \Leftrightarrow
		\sum_{i = 1}^m \left[ \norm{X_{i, :}} - \mu \right]_+ = \lambda.
		\label{eq:pwl-equations}
	\end{equation}
	Finally, we can use the algorithm proposed by~\cite{DSSC08} to
	solve for $\mu$ in~\eqref{eq:pwl-equations}. The algorithm
	involves a precomputation step for all row norms, achievable
	in time $O(mn)$, and a sorting step that takes time $O(m \log m)$.
	This completes the proof of the claim.
\end{proof}

\begin{lemma}
	\label{lemma:prox-tvnorm-surrogate}
	Let $M \in \Rbb^{m \times n}$ and write $\norm{M}_{1, 2} = \sum_{i = 1}^m \norm{M_{i, :}}$.
	Then for any $\mu > 0$,
	\begin{align}
		\prox_{\mu \norm{\cdot}_{1, 2}}(X)                  & =
		\begin{bmatrix}
			X_{1, :} \left(1 - \frac{\mu}{\max(\norm{X_{1, :}}, \mu)}\right) &
			\dots                                                            &
			X_{m, :} \left(1 - \frac{\mu}{\max(\norm{X_{m, :}}, \mu)}\right)
		\end{bmatrix}^{\T}
		\label{eq:prox-munorm-12}                               \\
		\bignorm{\prox_{\mu \norm{\cdot}_{1, 2}}(X)}_{1, 2} & =
		\sum_{i = 1}^m [\norm{X_{i, :}} - \mu]_+.
		\label{eq:prox-munorm-12-norm}
	\end{align}
\end{lemma}
\begin{proof}
	Note that $\norm{M}_{1, 2}$ is separable across rows. We write
	\begin{align*}
		\prox_{\mu \norm{\cdot}_{1, 2}}(M) & =
		\argmin_{X} \set{
			\frac{1}{2 \mu} \frobnorm{X - M}^2
			+ \norm{X}_{1, 2}
		}                                      \\
		                                   & =
		\argmin_{\set{X_{i, :}}_{i=1}^m}
		\sum_{i = 1}^m
		\set{
			\frac{1}{2 \mu} \norm{X_{i, :} - M_{i, :}}^2
			+ \norm{M_{i, :}}
		}                                      \\
		                                   & =
		\prox_{\mu \norm{\cdot}}(M_{1, :}) \times \dots \times
		\prox_{\mu \norm{\cdot}}(M_{m, :}).
	\end{align*}
	From standard results~\cite[Example 6.19]{Beck17}, we know that
	\[
		\prox_{\mu \norm{\cdot}}(M_{i, :}) = M_{i, :} \cdot \left(
		1 - \frac{\mu}{\max(\norm{M_{i, :}}, \mu)}
		\right).
	\]
	This yields the expression in~\eqref{eq:prox-munorm-12}. For~\eqref{eq:prox-munorm-12-norm},
	we first note that
	\[
		\norm{X_{i, :}} \left( 1 - \frac{\mu}{\max(\norm{X_{i, :}}, \mu)} \right) =
		\begin{cases}
			0,                     & \norm{X_{i, :}} < \mu, \\
			\norm{X_{i, :}} - \mu, & \text{otherwise}.
		\end{cases}
	\]
	Therefore, the sum of individual norms yields precisely the right-hand side of~\eqref{eq:prox-munorm-12-norm}.
\end{proof}

\begin{lemma}
	\label{lemma:proj-Jtv-graph}
	Let $J_{\mathsf{TV}}$ be defined in~\eqref{eq:dtv-def-1}-\eqref{eq:dtv-def-2}. Then
	\[
		\proj_{\gph(J_{\mathsf{TV}})}(x, z) = \begin{bmatrix}
			I_{n^2} & J_{\mathsf{TV}}^{\T} \\ J_{\mathsf{TV}} & -I_{2n^2}
		\end{bmatrix}^{-1}
		\begin{bmatrix}
			x + J_{\mathsf{TV}}^{\T} z \\
			\bm{0}_{2n^2}
		\end{bmatrix}.
	\]
	Moreover, it can be computed efficiently given a cached
	\texttt{LDLt} or Cholesky factorization.
\end{lemma}
\begin{proof}
	For an arbitrary linear operator $A$, we have~\citep{PB14}:
	\begin{align*}
		(x, y) = \proj_{\gph(A)}(\bar{x}, \bar{y}) \Leftrightarrow
		\begin{bmatrix}
			I & A^{\T} \\
			A & -I
		\end{bmatrix}
		\begin{bmatrix} x \\ y \end{bmatrix}
		 & =
		\begin{bmatrix}
			\bar{x} + A^{\T} \bar{y} \\
			\bm{0}
		\end{bmatrix}.
	\end{align*}
	When $A = J_{\mathsf{TV}}$ and viewing $J_{\mathsf{TV}}$
	as a linear operator from $\Rbb^{n^2} \to \Rbb^{2n^2}$, we have
	\begin{align*}
		\mathsf{nnz}\left(
		\begin{bmatrix}
				I_{n^2}         & J_{\mathsf{TV}}^{\T} \\
				J_{\mathsf{TV}} & -I_{2n^2}
			\end{bmatrix}
		\right) = O(n^2),
	\end{align*}
	which means that the matrix on the left-hand side can be factorized using
	an \texttt{LDLt} factorization. Alternatively, eliminating $y$ from the system
	yields the updates
	\begin{align*}
		x & = (I + J_{\mathsf{TV}}^{\T} J_{\mathsf{TV}})^{-1} (\bar{x} + J_{\mathsf{TV}}^{\T} \bar{y}),
		\quad y = J_{\mathsf{TV}}(x).
	\end{align*}
	In the above, we can compute a Cholesky factorization of $I + J_{\mathsf{TV}}^{\T} J_{\mathsf{TV}}$
	and solve the resulting linear systems with low computational effort.
\end{proof}

\section{An $f_{\star}$-adaptive variant of~\cref{alg:polyaksgm}}
\label{sec:appendix:polyaksgm-noopt}
We describe a variant of~\cref{alg:polyaksgm} for problems
with unknown optimal value $f_{\star}$, such as CT problems with noisy measurements,
due to~\citet{HK19}; see~\cref{alg:polyaksgm-noopt} below.

\begin{algorithm}
	\caption{\texttt{PolyakSGM-NoOpt}~\cite[Algorithm 2]{HK19}}
	\begin{algorithmic}[1]
		\State \textbf{Input}: $x_0 \in \Rbb^{d}$, lower bound $\widetilde{f}_0 \leq f_{\star}$,
		$\eta \in (0, 1)$, budgets $T_{\textsf{inner}}, T_{\textsf{outer}} \in \Nbb$.
		\For{$t = 1, \dots, T_{\textsf{outer}}$}
		\State Compute $\widehat{x}_{t} := \texttt{PolyakSGM}(x_0, \frac{\eta}{2}, T_{\textsf{inner}}, \widetilde{f}_{t-1})$
		\Comment{Alg.~\ref{alg:polyaksgm} with opt. value estimate $\widetilde{f}_{t-1}$}
		\State Update $\widetilde{f}_{t} := \frac{1}{2} \left( f(\widehat{x}_t) + \widetilde{f}_{t-1} \right)$
		\EndFor
		\State \Return $\argmin_{z} \set{f(z) \mid z \in \set{\widehat{x}_{i}}_{i = 1}^{T_{\mathsf{outer}}}}$.
	\end{algorithmic}
	\label{alg:polyaksgm-noopt}
\end{algorithm}

At a high-level,~\cref{alg:polyaksgm-noopt} maintains
an estimate $\widetilde{f}_{t}$ of the optimal value $f_{\star}$, which is updated using the
objective value found by successive calls to~\cref{alg:polyaksgm}.
In particular,~\citet{HK19} show that if $\widetilde{f}_{0} \leq f_{\star}$, running~\cref{alg:polyaksgm-noopt} is
equivalent to running~\cref{alg:polyaksgm} with
the ``correct'' value of $f_{\star}$, stepsize $\frac{\eta}{2}$, and budget $T_{\mathsf{inner}}$.
Their analysis is for the case where $f$ is convex (and thus $\eta = 1$ is optimal); while
it can be extended to our setting, under the restriction $\eta \leq \frac{1}{\condnum}$,
we do not pursue this generalization here.

\end{document}